\documentclass[12pt]{amsart}
\usepackage{color}
\usepackage[all]{xy}
\usepackage{amssymb,amsmath}

\usepackage{mathrsfs}
\usepackage{eucal}

\usepackage{setspace}

\date{December 27, 2014}

\calclayout
\newtheorem{dummy}{anything}[section]
\newtheorem{theorem}[dummy]{Theorem}
\newtheorem*{thma}{Theorem A}
\newtheorem*{thmb}{Corollary B}
\newtheorem*{thmc}{Theorem C}
\newtheorem{lemma}[dummy]{Lemma}
\newtheorem{proposition}[dummy]{Proposition}
\newtheorem{corollary}[dummy]{Corollary}

\theoremstyle{definition}
\newtheorem{definition}[dummy]{Definition}
  \newtheorem{example}[dummy]{Example}
  
  \newtheorem{remark}[dummy]{Remark}
  \newtheorem*{Remark}{Remark}
   \newtheorem*{Definition}{Definition}
   
    \newtheorem*{question}{Question}
  \newtheorem*{acknowledgement}{Acknowledgement}

\newcommand{\cF}{\mathcal F}

\newcommand{\cH}{\mathcal H}

\newcommand{\cJ}{\mathcal J}

\newcommand{\cS}{\mathcal S}
\newcommand{\cZ}{E}

\newcommand{\bC}{\mathbf C}
\newcommand{\bD}{\mathbf D}
\newcommand{\bF}{\mathbf F}
\newcommand{\bZ}{\mathbf Z}

\newcommand{\bN}{\mathbf N}
\newcommand{\bP}{\mathbf P}
\newcommand{\bV}{\mathbf V}

\newcommand{\bbC}{\mathbb C}
\newcommand{\bbF}{\mathbb F}

\newcommand{\bbZ}{\mathbb Z}

 \DeclareMathOperator{\Ind}{Ind}
\DeclareMathOperator{\Res}{Res}
\DeclareMathOperator{\Iso}{\mathfrak{Iso}}
\DeclareMathOperator{\Mor}{Mor}
 \DeclareMathOperator{\Map}{Map}
\DeclareMathOperator{\ind}{Ind}
 \DeclareMathOperator{\Dim}{Dim}
\DeclareMathOperator{\HomDim}{hDim}
\DeclareMathOperator{\hdim}{hdim}
  \DeclareMathOperator{\rk}{rank}
 \DeclareMathOperator{\rank}{rank}
 \DeclareMathOperator{\Rep}{Rep}

\newcommand{\la}{\langle}
\newcommand{\ra}{\rangle}
\newcommand{\vv}{\, | \,}
\newcommand{\bd}{\partial}

\newcommand\Fp{\bbF_p}
\def\bZp{\bZ_{(p)}}
\def\bZq{\bZ_{(q)}}
\def\G{\varGamma}

\newcommand{\leftexp}[2]{{\vphantom{#2}}^{ #1}{\hskip-1pt#2}}

\DeclareMathOperator{\Or}{Or}
\newcommand{\Sub}{{\cS}(G)}

\newcommand\RG{R\G}
\newcommand\ZG{\bZ\G}
\newcommand\OrG{\Or _{\cF}G}
\newcommand\un{\underline}
\newcommand{\up}{^{(p)}}
\newcommand{\uq}{^{(q)}}
\newcommand\uR[1]{R[{#1}^{\, \textbf{?}\,}]}
\newcommand\uC[1]{\CC({#1}^{\textbf{?}};R)}
\newcommand\uCZ[1]{\CC({#1}^{\textbf{?}};\bZ)}
\newcommand\CC{\bC}
\newcommand\DD{\bD}
\newcommand\NN{\bN}
\newcommand\PP{\bP}

\newcommand\VV{\bV}

\newcommand{\ZZ}{\mathbf Z}
\newcommand\uH[1]{H_*({#1}^{\textbf{?}};R)}
\DeclareMathOperator{\Qd}{Qd}
\newcommand{\bigast}{\divideontimes}
\newcommand\nr[2]{\medskip\noindent\textbf{#1:\ #2.}}
\newcommand{\PG}{\cS_G}

\newcommand{\gls}{gorenstein-lyons-solomon1998}

\begin{document}

\title{Group actions on Spheres with rank one isotropy}
\author{Ian Hambleton}
\author{Erg\"un Yal\c c\i n}

\address{Department of Mathematics, McMaster University,
Hamilton, Ontario L8S 4K1, Canada}

\email{hambleton@mcmaster.ca }

\address{Department of Mathematics, Bilkent University,
06800 Bilkent, Ankara, Turkey}

\email{yalcine@fen.bilkent.edu.tr }

\thanks{Research partially supported by NSERC Discovery Grant A4000. The second author is partially supported by T\" UB\. ITAK-TBAG/110T712.}

\begin{abstract}
Let $G$ be a rank two finite group, and let $\cH$ denote the family of all rank one $p$-subgroups of $G$ for which  $\rk_p(G)=2$. 
We show that a rank two finite group $G$ which satisfies certain explicit group-theoretic conditions  admits a finite $G$-CW-complex $X\simeq S^n$ with isotropy in $\cH$, whose fixed sets are homotopy spheres. Our construction provides an infinite family of new non-linear $G$-CW-complex examples. 
\end{abstract}

\maketitle
 
\section{Introduction}
\label{sect:introduction}

Let $G$ be a finite group. The unit spheres $S(V)$ in  finite-dimensional orthogonal representations of $G$ provide the basic examples of smooth linear $G$-actions on spheres. These linear  actions satisfy a number of special constraints on the dimensions of fixed sets and the structure of the isotropy subgroups, arising from character theory.  However, such constraints do not hold in general for smooth $G$-actions on spheres, unless $G$ has prime power order (see  \cite{dotzel-hamrick1}). 
Our goal in this series of papers is to construct new examples of smooth \emph{non-linear} finite group actions on spheres, with prescribed isotropy.

 In the first paper of this series  \cite{h-yalcin3}, we studied group actions on spheres in the setting of 
\emph{$G$-homotopy representations},
introduced by tom Dieck (see \cite[Definition 10.1]{tomDieck2}). These are finite (or more generally finite dimensional) $G$-CW-complexes $X$ satisfying the property that for each $H\leq G$, the fixed point set $X^H$ is homotopy equivalent to a sphere $S^{n(H)}$ where $n(H)=\dim X^H$.  We introduced algebraic homotopy representations as suitable chain complexes over the orbit category and proved a realization theorem for these algebraic models. 

We say that $G$ has \emph{rank $k$} if it contains a subgroup isomorphic to $(\bZ/p)^k$, for some prime $p$, but no subgroup $(\bZ/p)^{k+1}$, for any prime $p$. In this paper,  we use chain complex methods to study the following problem, as the next step towards smooth actions.  

\begin{question}\label{ques:mainquestion} For which finite groups  $G$, does there exist a finite $G$-CW-complex $X\simeq
S^n$ with all isotropy subgroups of rank one ~?
\end{question}

The isotropy assumption implies that $G$ must have rank $\leq 2$, by P.~A.~Smith theory  (see Corollary \ref{cor:ranktwo}).    
Since every rank one finite group can act freely on a finite
complex homotopy equivalent to a sphere (Swan \cite{swan1}), we will restrict  to groups of rank two.  

There is another group theoretical \emph{necessary} condition related to fusion properties of the Sylow subgroups. This condition involves the rank two finite group $\Qd(p)$ which is the group defined as the semidirect product 
$$\Qd(p) = (\bZ/ p \times \bZ / p)\rtimes SL_2(p)$$
with the obvious action of $SL_2(p)$ on $\bZ / p \times \bZ /p$. In his thesis, \" Unl\" u \cite[Theorem 3.3]{unlu1} showed that $\Qd(p)$ does not act on a finite CW-complex $X\simeq S^n$ with rank one isotropy.  This means that any rank two finite group which includes $\Qd(p)$ as a subgroup cannot admit such actions. 
 
More generally, we say $\Qd(p)$ is $p'$-\emph{involved in $G$} if there exists a subgroup $K \leq G$, of order prime to $p$, such that $N_G(K)/K$ contains a subgroup isomorphic to $\Qd(p)$. The argument given by \" Unl\" u in \cite[Theorem 3.3]{unlu1} can be extended easily to obtain the stronger necessary condition  
(see Proposition \ref{pro:involveQdp}):

\medskip
\noindent ($\bigast$). \emph{Suppose that there exists a finite $G$-CW-complex $X\simeq S^n$ with rank one isotropy. Then $\Qd(p)$ is not $p'$-involved in $G$,  for any odd prime $p$.}

\medskip

In the other direction, the Sylow subgroups of 
rank two finite groups which do not $p'$-involve $\Qd(p)$, for $p$ odd (sometimes called \emph{$\Qd(p)$-free} groups),  have some interesting complex representations.  

\begin{Definition} 
A finite group $G$ has a $p$-\emph{effective character} if each $p$-Sylow subgroup $G_p$ of $G$ has a character $\chi\colon G_p \to \bbC$ 
which (i) \emph{respects fusion in $G$}, meaning that $\chi (gxg^{-1})=\chi(x)$ whenever $gxg^{-1} \in G_p$ for some $g \in G$ and $x \in G_p$, and (ii) satisfies $\la \chi |_E, 1_E\ra = 0$ for each elementary abelian $p$-subgroup $E$ of $G$ with $\rank E = \rank_p G$. 
\end{Definition}

Jackson \cite[Theorem~47]{jackson1} proved that a rank two group $G$ has a $p$-effective character if and only if $p=2$, or $p$ is odd and $G$ is  $\Qd(p)$-free. We will use these characters to obtain finite $G$-homotopy representations with rank one prime power isotropy, assuming an additional closure property at the prime $p=2$.


\begin{Definition}\label{def:IntProp} A finite group $G$ has the   \emph{rank one intersection property} if for every pair $H,K \leq G$ of rank one $2$-subgroups  such that $H \cap K\neq 1$, the subgroup $\la H , K\ra$ generated by $H$ and $K$ is a $2$-group. We say that $G$ is \emph{ $2$-regular}
if (i)  $\Omega_1(Z(G_2))$ is strongly closed in $G_2$ with respect to $G$, or (ii) $G$ has the  rank one intersection property.
\end{Definition}

Let  $\cF$ be a family of subgroups of $G$ closed under conjugation and taking subgroups. For constructing group actions on CW-complexes with isotropy in the family $\cF$, a good algebraic approach is to consider projective chain complexes over the orbit category relative to the family $\cF$ (see \cite{hpy1}, \cite{h-yalcin3}). 

\smallskip
Let $\PG$  denote the set of primes $p$ such that $\rk_p(G)=2$. Let $\cH_p$ denote the family of all rank one  $p$-subgroups $H \leq G$, for $p \in \PG$, and  let $\cH = \bigcup\, \{ H \in \cH_p \vv p\in \PG\}$.
Our main result is the following:

\begin{thma}\label{thm:main} Let $G$ be
rank two finite group satisfying  the following two conditions: 
\begin{enumerate}
\item $G$ is    $2$-regular if $2 \in \PG$, and $G$ is $\Qd(p)$-free for all $p\in \PG$ with $p>2$;  
\item  If $1 \neq H \in \cH_p$, then $\rk_q(N_G(H)/H) \leq 1$ for every prime $q\neq p$.  
\end{enumerate}
Then there exists a finite    $G$-homotopy representation $X$ with isotropy in $\cH$. 
\end{thma}

Theorem A is an extension of our earlier joint work with Semra Pamuk  \cite{hpy1} where we have shown that the first non-linear example, the permutation group $G = S_5$ of order 120, admits a finite $G$-CW-complex $X \simeq S^n$ with rank one isotropy. Theorem A gives a new proof of this earlier result, by a more systematic method:  for $G=S_5$, the set $\PG$ includes only the  prime $2$ and it can be easily seen that $G$ satisfies the rank one intersection property. 
 The second condition above also holds since all $p$-Sylow subgroups of $S_5$ for odd primes are cyclic.  More generally, we have:

\begin{thmb}\label{cor:main}
Let $p$ be a fixed prime and $G$ be a finite group such that $\rk _p (G)=2$, and $\rk _q (G)=1$ for every prime $q\neq p$. If $G$ is $\Qd(p)$-free when $p>2$, and  
$G$ is    $2$-regular when $p=2$,
then there exists a finite    $G$-homotopy representation $X$ with rank one $p$-group isotropy. 
\end{thmb} 

As a consequence of Corollary B, we 
 show  that $G=PSL_2(q)$, where $q\geq 5$ is a prime, 
admits a    $G$-homotopy representation with cyclic $2$-subgroup isotropy. Note that none of the simple groups $PSL_2(q)$, $q >7$, admit orthogonal representation spheres with rank one isotropy (see Section \ref{sec:simple}), so the actions we construct provide an infinite family of new examples of \emph{non-linear} actions.

More generally, using Theorem A, we obtain many new non-linear $G$-CW-complex examples.  
In particular, we show that the alternating group $A_6$ admits finite $G$-CW-complexes $X \simeq S^n$ with rank one isotropy  (see Example \ref{ex:asix}). We also discuss  why $G=A_7$ can not admit such actions if we require $X$ to be an $G$-homotopy  representation with rank one prime power isotropy (see Example \ref{ex:aseven}). In fact we show exactly which  of the rank two simple groups (see the list in \cite[p.423]{adem-smith1}) can admit such actions.

\begin{thmc} Let $G$ be a finite simple group of rank two. Then there exists a finite   $G$-homotopy representation with rank one isotropy of prime power order  if and only if 
$G$ is one of the following: (i) $PSL_2(q)$, $q\geq 5$, (ii)  $PSL_2 (q^2)$, $q \geq 3$, 
(iii) $ PSU_3(3)$, or (iv) $ PSU_3(4)$.
\end{thmc}

We remark that  $G = PSL_3(q)$, $q$ odd, and $G = PSU_3(q)$, with $9 \mid (q+1)$, are the rank two simple groups that are not $\Qd(p)$-free\footnote{This case seems to have been overlooked in \cite[p.430]{adem-smith1}} at some odd prime. The remaining simple groups $G=PSU_3(q)$, $q\geq 5$, are eliminated by  the Borel-Smith conditions (see Section \ref{sec:simple}). 
 The groups $PSU_3(3)$ and $PSU_3(4)$ have a linear actions on  spheres with rank one prime power isotropy. We note that the group $G=PSU_3(3)$  does not satisfy the rank one intersection property
 (see Example \ref{ex:PSU(3,3)}). 
  
In Section \ref{sec:examples},  we give the motivation for condition (ii) in Theorem A on the $q$-rank of the normalizer quotients $N_G(H)/H$ for all the subgroups $H \in \cH$. It is used
 in a crucial way (at the algebraic level) in the construction of a
  finite $G$-CW-complex
    $X\simeq S^n$ with rank one isotropy in $\cH$, which is a    $G$-homotopy representation.
 However, condition (ii) in Theorem A is actually \emph{necessary} only for the subgroups $H \in \cH$ such that $X^H \neq \emptyset$ (see Remark \ref{rem:necessity2}), but not, in general,  for all rank one $p$-subgroups (see Example \ref{ex:aseven}). 
Determining the precise list of necessary and sufficient conditions is still an open problem.

We will obtain Theorem A from a more general technical result, 
Theorem \ref{thm:maintech}, 
which accepts as input a compatible collection of representations defined on all $p$-subgroups of $G$, for a given set of primes (see Definition \ref{def:sylowrep}), 
and produces a finite $G$-CW complex. 
In principle, Theorem \ref{thm:maintech} can be used to construct other interesting non-linear examples for finite groups with specified $p$-group isotropy. 

\medskip
 Here is a brief outline of the paper. We denote the orbit category relative to a family $\cF$  by $\G _G =\Or _{\cF} G$,   and construct projective chain complexes over $\RG_G$ for various $p$-local coefficient rings $R=\bZ _{(p)}$. 
 To prove Theorem \ref{thm:maintech}, we first introduce  \emph{algebraic homotopy representations} (see Definition \ref{def:algrep}), as chain complexes over $\RG_G$
 satisfying algebraic versions of the conditions found in tom Dieck's 
 \emph{$G$-homotopy representations} (see \cite[II.10.1]{tomDieck2}, \cite{dotzel-hamrick1}, 
and Remark \ref{simplyconnected}). 
 In Section \ref{sect:AlgHomRep} we summarize the results of \cite{h-yalcin3} which show that the conditions in Definition \ref{def:algrep} lead to 
 necessary and sufficient conditions for  a chain complex over $\RG_G$ to be homotopy equivalent to a chain complex of a    $G$-homotopy representation (see Theorem \ref{thm:tightness}).

 In Section  \ref{sec:prime p}, we construct $p$-local chain complexes where  the isotropy subgroups are $p$-groups.   In Section \ref{sec:other primes}, we add homology to these local models so that these modified local complexes $\CC \up$ all have exactly the same dimension function.  Results established in \cite{hpy1} are used to  glue these algebraic complexes together over $\bZ\G_G$, and  then  to eliminate a finiteness obstruction. In Section \ref{sec:main theorem}  we combine these ingredients to
give a complete proof for Theorem \ref{thm:maintech} and Theorem A.  In Section \ref{sec:examples}, we discuss the necessity of the conditions in Theorem A and provide  a nonlinear action for the group $G=A_6$.
%
We study the rank two simple groups and prove Theorem C in Section \ref{sec:simple}. 

\begin{Remark}
One motivation for this project is the work of Adem-Smith \cite{adem-smith1} and Jackson \cite{jackson1} on the existence of free actions of finite groups on a product of two spheres. There is an interesting set of conditions related to this problem. In the following statements, $G$ denotes a finite group of rank two.
\begin{enumerate}\addtolength{\itemsep}{0.3\baselineskip}
\item $G$ acts on a finite complex $X$ homotopy equivalent to a sphere, with rank one isotropy.
\item  $G$ acts with rank one isotropy  on a finite dimensional complex $X$ which has a 
mod $p$ homology of a sphere. 
\item $G$ does not $p'$-involve $\Qd(p)$, for $p$ an odd prime.
\item $G$ has a $p$-effective character $\chi : G_p \to \bbC$. 

\item There exists a spherical fibration $Y \to BG$, such that the total space $Y$ has periodic cohomology.
\item $G$ acts freely on a finite complex homotopy equivalent to a product of two spheres.
\end{enumerate}
The implications $(i)\Rightarrow (i+1)$ hold for this list (suitably interpreted), 
where $(i) \Rightarrow (ii)$ is clear (for each prime $p$), and $(ii) \Rightarrow (iii)$ is our Proposition \ref{pro:involveQdp}. The implication $(iii) \Leftrightarrow (iv)$ is due to Jackson \cite[Theorem 47]{jackson1},  using  \cite[Theorem 44]{jackson1} to show that $G$ 
always has a $2$-effective character. 

If condition $(iv)$ holds for all the primes dividing the order of $G$, then condition $(v)$ holds. This needs some explanation.  First, the existence of a spherical fibration $ Y \to BG$ classified by $\varphi\colon BG \to BU(n)$, with $p$-effective Euler class $\beta(\varphi) \in H^n(G;\bZ)$ for all primes $p$,  was proved by Jackson  \cite{jackson2}, \cite[Theorem 16]{jackson1}. By construction, for each elementary abelian $p$-subgoup $E$ of $G$ with $\rk E = \rk_p G$, there exists a unitary representation $\lambda \colon E \to U(n)$ such that $\varphi_E = B\lambda$ and $\la \lambda, 1_E\ra = 0$ (see  \cite[Definition 11]{jackson1}). Adem and Smith  \cite[Definition 4.3]{adem-smith1} give an equivalent definition of a $p$-effective cohomology class $\beta \in H^n(G;\bZ)$ as a class for which the complexity $cx_G(L_\beta\otimes \Fp)=1$ (see Benson \cite[Chap.~5]{benson2}). It follows from  \cite[5.10.4]{benson2} that $L_{\beta(\varphi)} \otimes \Fp$ is a  periodic module,  and hence cup product with a periodicity generator $\alpha$ for this module gives the periodicity of $H^*(Y;\Fp)$ in high dimensions. Therefore $Y$ has periodic cohomology in the sense of Adem-Smith \cite[Definition 1.1]{adem-smith1}.
Finally, $(v) \Rightarrow (vi)$ follows from the main results of Adem-Smith \cite[Theorems 1.2, 3.6]{adem-smith1}.

The reverse implications are mostly unknown. For example, it is not known whether $\Qd(p)$ itself can act freely on a product of two spheres. In \cite[Theorem 47]{jackson1} it was claimed that $(iii) \Rightarrow (i)$, but the ``proof" seems to confuse homotopy actions with finite 
$G$-CW complexes. However, we show in Corollary \ref{cor:p-local sphereA} that $(iii) \Rightarrow (ii)$. Finding new criteria for the implication $(iii) \Rightarrow (i)$  is the subject of this paper.
\end{Remark}

\begin{acknowledgement} The authors would like to thank Alejandro Adem, Jesper Grodal, Radha Kessar and Assaf Libman for helpful communications, and to Ron Solomon for his invaluable help in providing detailed information about the rank two simple groups. 
\end{acknowledgement}

\section{Algebraic homotopy representations}
\label{sect:AlgHomRep}

In transformation group theory, a $G$-CW-complex $X$ is called a    \emph{$G$-homotopy representation} if it has the property that $X^H$ is homotopy equivalent to the sphere $S^{n(H)}$ where $n(H)=\dim X^H$,  for every $H\leq G$ (see tom Dieck  \cite[Section II.10]{tomDieck2}).

In this section we summarize the results of \cite{h-yalcin3} which gives the definition and main properties of a suitable algebraic analogue, called \emph{algebraic homotopy representations}.

Let $G$ be a finite group and $\cF$ be a family of subgroups of $G$ which is closed under conjugations and taking subgroups. 
The orbit category $\OrG$ is defined as the category whose objects are orbits of type
$G/K$, with $K \in \cF$, and where the morphisms from $G/K$ to $G/L$
are given by $G$-maps:
$$ \Mor _{\OrG} (G/K , G/L ) = \Map _G (G/K, G/L ).$$
The category $\G _G= \OrG$ is a small category, and we can consider the
module category over $\G _G$. Let $R$ be a commutative ring with unity.
A (right) $\RG _G$-module $M$ is a contravariant functor from $\G _G$ to
the category of $R$-modules. We denote the $R$-module $M(G/K)$
simply by $M(K)$ and write $M(f)\colon  M(L) \to M(K)$  for a $G$-map $f\colon 
G/K \to G/L$.  The further details about the properties of modules over the orbit category, such as the definitions of free and projective modules, can be found in \cite{hpy1} (see also L\"uck \cite[\S 9,\S 17]{lueck3} and tom Dieck \cite[\S 10-11]{tomDieck2}).

We will consider chain complexes $\CC$ of $\RG _G$-modules, such that $\CC _i=0$ for $i<0$. We call a chain complex $\CC$ \emph{projective} (resp.~\emph{free}) if for all $i\geq 0$,
the modules $\CC_i$ are projective (resp.~free). We say that a chain complex $\CC$ is \emph{finite} if $\CC_i = 0$ for $i > n$, and the chain modules $\CC_i$ are all finitely-generated $\RG_G$-modules.

Given a $G$-CW-complex $X$, associated to it,
there is a chain complex of $\RG_G$-modules
$$\uC{X}:   \quad  \cdots \to \uR{X_n}
\xrightarrow{\bd_{n}}  \uR{X_{n-1}} \rightarrow
\cdots\xrightarrow{\bd_{1}}  \uR{X_0} \to 0$$ where $X_i$ denotes
the set of $i$-dimensional cells in $X$ and $\uR{X_i}$ is the
$\RG_G$-module defined by $ \uR{X_i}(H)= R[X_i^H]$. We denote the
homology of this complex by $\uH{X}$. If the family $\cF$ includes
the isotropy subgroups of $X$, then the complex $\uC{X}$ is a chain complex of free $\RG_G$-modules.  

The \emph{dimension function} of a finite dimensional chain complex $\CC$ over $\RG_G$ is defined as the function $\Dim \CC \colon  \Sub \to \bZ$, where $\Sub$ denotes the family of \emph{all} subgroups of $G$,  given by
 $$(\Dim \CC )(H)=\dim \CC (H)$$ for all $H \in \cF$.
If $\CC (H)$ is the zero complex or if $H$ is  a subgroup such that $H \not \in \cF$, then we define $(\Dim \CC)(H)=-1$. 
The dimension function $\Dim \CC$  is constant on conjugacy classes (a super class function).
 In a similar way, we can define the \emph{homological dimension function}  $\HomDim \CC \colon  \Sub \to \bZ$ of a chain complex $\CC$ of $\RG_G$-modules. 

We call a function $\un{n}\colon \Sub \to  \bZ$ \emph{monotone} if it
satisfies the property that $\un{n}(K) \leqslant \un{n}(H)$ whenever
$(H) \leq (K)$. We say that a monotone function  $\un{n}$ is
\emph{strictly monotone} if $\un{n}(K) < \un{n}(H)$, whenever $(H)< (K)$.  
We have the following:

\begin{lemma}[{\cite[Lemma 2.6]{h-yalcin3}}
] \label{lem:monotone}
The dimension function of a projective chain
complex of $\RG_G$-modules is a monotone function.
\end{lemma}
      
\begin{definition} We say a  chain complex $\CC$ of $\RG_G$-modules 
is {\it tight at $H\in \cF$} if $$\Dim \CC (H)=\hdim \CC (H).$$  We call a chain complex of $\RG _G$-modules {\it tight} if it is tight at every $H\in \cF$. 
\end{definition}

We are particularly interested in chain complexes which have the homology of a sphere when evaluated at every $K \in \cF$. 
Let $\un{n}$ be a super class function 
\emph{supported 
on $\cF$}, meaning that $\un{n}(H) = -1$ for $H\notin \cF$, 
and let $\CC$ be a chain complex over $\RG_G$.
 We say that $\CC$ is
  an \emph{$R$-homology $\un {n}$-sphere}
  (see  \cite[Definition 2.7]{h-yalcin3})
 if  the reduced homology of $\CC (K)$ is the same as the reduced homology of an $\un{n}(K)$-sphere (with coefficients in $R$) for all $K \in \cF$. Here the \emph{reduced homology} is the homology of an augmented chain complex $\varepsilon\colon \CC \to \un{R}$, with $\varepsilon(H)$ surjective for all $H\in \cF$ such that $\CC(H) \neq 0$.

In \cite[II.10]{tomDieck2}, there is a list of properties that are satisfied by   
 $G$-homotopy representations. We will use algebraic versions of these properties to define an analogous notion for chain complexes.

\begin{definition}[{\cite[Definition 2.8]{h-yalcin3}}]
\label{def:algrep} Let $\CC$ be a finite projective  chain complex over $\RG_G$, which is an $R$-homology $\un{n}$-sphere.
We say $\CC$ is an \emph{algebraic homotopy representation} (over $R$) if 
\begin{enumerate}
\item The function $\un{n}$ is a monotone
function.
\item If $H,K \in \cF$ are such that $n=\un{n}(K)=\un{n}(H)$,
then for every $G$-map $f \colon  G/H \to G/K$ the induced map $\CC(f)\colon \CC (K) \to \CC (H)$ is an $R$-homology isomorphism.
\item Suppose $H, K, L\in \cF$ are such that $H \leq K,L$ and
let $M=\langle K, L \rangle$ be the subgroup of $G$ generated by $K$ and $L$. If
$n=\un{n}(H)=\un{n}(K)=\un{n}(L) >-1$, then $M \in \cF$ and
$n=\un{n}(M)$.
\end{enumerate}
\end{definition}

If a dimension function $\un{n}$ satisfies condition (iii) of Definition \ref{def:algrep}, then we say it has the  \emph{closure property}. Such dimension functions have an important maximality property. 

\begin{proposition}[{\cite[Proposition 2.9]{h-yalcin3}}]
 \label{cor:max} Let $\CC$ be a projective chain complex of $\RG_G$-modules, which is an $R$-homology $\un{n}$-sphere. 
If $\un{n}$ satisfies the closure property, then the set of subgroups $\cF_H = \{ K \in \cF \vv (H) \leq (K),\ \un{n}(K) = \un{n}(H)> -1\}$ has a unique maximal element, up to conjugation. 
\end{proposition}

In the remainder of this section we will assume that $R$ is a principal ideal domain. The main examples for us are $R =\bZp$ or $R=\bZ$.

 \begin{theorem}[{\cite[Theorem A]{h-yalcin3}}]
Let $\CC$ be a finite free chain complex of $\RG_G$-modules which is an
$R$-homology
$\un{n}$-sphere. Then $\CC$ is chain homotopy equivalent to a finite
free chain complex $\DD$ which is tight if and only if $\CC$ is an algebraic homotopy representation. 
\end{theorem}

When these conditions hold for $R = \bbZ$, then we apply 
\cite[Theorem 8.10]{hpy1}, \cite{pamuks1}
 to obtain a geometric realization result.

\begin{theorem}[{\cite[Corollary B]{h-yalcin3}}]\label{thm:tightness} 
Let $\CC$ be a finite free chain complex of $\ZG_G$-modules which is a
homology $\un{n}$-sphere. If $\CC$ is an  algebraic homotopy representation,
and $\un{n}
(K) \geq 3$ for all $K \in \cF$, then there is a finite $G$-CW-complex $X$ such that $\uCZ{X}$ is chain homotopy equivalent to
$\CC$ as chain complexes of $\ZG_G$-modules. 
\end{theorem}

\begin{remark}\label{simplyconnected} 
The construction actually produces a finite $G$-CW-complex $X$ such that all the non-empty fixed sets $X^H$ are simply-connected, and with trivial action of $W_G(H) = N_G(H)/H$  on the homology of $X^H$. Therefore $X$ will be an  \emph{oriented}    $G$-homotopy representation (in the sense of tom Dieck). 
\end{remark}

\section{Construction of the preliminary local models}\label{sec:prime p}

Our main technical tool  is provided by Theorem \ref{thm:maintech}, which gives a method for constructing finite $G$-CW-complexes, with isotropy in a given family.  This theorem will be proved by applying the realization statement of Theorem \ref{thm:tightness}. To construct a suitable finite free chain complex $\CC$ over $\ZG_G$, we work one prime at a time to construct local models $\CC\up$, and then apply the glueing method for chain complexes developed in \cite[Theorem 6.7]{hpy1}.

The main input of Theorem \ref{thm:maintech}  is a compatible collection of
unitary representations for the $p$-subgroups of $G$. We give the precise definition in  a more general setting.

\begin{definition}\label{def:sylowrep} Let $\cF$ be a family of subgroups of $G$ and $n$ be a fixed integer.
We say that $\VV(\cF)$ is an \emph{$\cF$-representation} for $G$ of dimension $n$, if 
$\VV(\cF) = \{ V_H \in \Rep(H, U(n)) \vv H \in \cF\}$ is a compatible collection
 of (non-zero) 
unitary $H$-representations.
The collection is \emph{compatible} if $f^* (V_K) \cong V_H$ for every $G$-map $f\colon  G/H \to G/K$.  
\end{definition}

For any finite $G$-CW-complex $X$, we let $\Iso(X) = \{K \leq G\vv X^K \neq \emptyset\}$ denote the \emph{isotropy family} of the $G$-action on $X$. Note that this is the smallest family closed under conjugation and taking subgroups, which includes all the isotropy subgroups of $X$. This suggests the following notation:

\begin{definition} Let $\VV(\cF)$ be an $\cF$-representation for $G$. We let 
$$\Iso(\VV(\cF)) = \{ L \leq H \vv S(V_H)^L \neq \emptyset, \text{\ for some\ } V_H \in \VV(\cF)\}$$
denote the \emph{isotropy family} of $\VV(\cF)$. We note that $\Iso(\VV(\cF))$ is a sub-family of $\cF$. 
\end{definition}

\begin{example}\label{ex:pgroupfamily}
Our first example for these definitions will be a compatible collection of  representations for the family $\cF_p$ of all $p$-subgroups, with $p$ a fixed prime dividing the order of $G$. In this case, an $\cF_p$-representation $\VV(\cF_p)$ can be constructed from a suitable representation $V_p \in \Rep(P, U(n))$, where $P$ denotes a $p$-Sylow subgroup of $G$. The representations $V_H$ can be constructed for all $H \in \cF_p$, by extending $V_p$ to conjugate $p$-Sylow subgroups and by restriction to subgroups. To ensure a compatible collection $\{V_H\}$, we assume that $V_p$ \emph{respects fusion} in $G$, meaning that $\chi_p (gxg^{-1})=\chi_p(x)$ for the corresponding character $\chi_p$, whenever $gxg^{-1}\in P$ for some
$g\in G$ and $x\in P$.
\end{example}

We will now specify an isotropy family $\cJ$ that will be used in the rest of the paper.
\begin{definition}\label{def:allprimes}
Let $\{ \VV(\cF_p) \vv p \in \PG\}$ be a collection of $\cF_p$-representations, for a set $\PG$  of primes dividing the order of $G$. Let  $\cJ_p = \Iso(\VV(\cF_p))$ and  $\cJ = \bigcup\{ \cJ_p \vv p \in \PG\}$ denote the  isotropy families.
\end{definition}

We note that $\cJ$ contains no isotropy subgroups of composite order, since each $\cJ_p$ is a family of $p$-subgroups. 
 Let $\G_G=\Or _{\cJ} G$ and $\G_G(p)$ denote the orbit category $\Or _{\cJ_p} G$ over the family $\cJ_p$. A chain complex $\CC$ over $\RG_G(p)$ can always be considered as a complex of $\RG_G$-modules,  by taking the values $\CC(H)$ at subgroups  $H \not \in \cJ_p$ as zero complexes.

 In this section we construct a $p$-local chain complex $\CC\up(0)$ over $\RG_G(p)$, for $R = \bZp$, which we call a \emph{preliminary local model} (see Definition \ref{def:prelocal}).  From this construction we will obtain a dimension function $\un{n}\up \colon \cJ_p \to \bZ$. By taking joins we can assume that these dimension functions have common value at $H=1$. In the next section, these preliminary local models will be ``improved" at each prime $p$ by adding homology as specified by the dimension functions $\un{n}\uq\colon \cJ_q \to \bZ$, for all $q\in \PG$ with $q \neq p$.  The resulting complexes $\CC\up$ over the orbit category $\RG_G$ will all have the same dimension function 
$$\un{n} = \bigcup\, \{\un{n}\up \vv p \in \PG  \} \colon \cJ \to \bZ,$$ and satisfy conditions needed for the glueing method.

\begin{proposition}\label{pro:p-local sphere} Let $G$ be a finite group, and let  $\VV(\cF_p)$  be an $\cF_p$-representation for $G$ for some $p \in \PG$. Then there exists a   finite-dimensional $G$-CW-complex $\cZ$, with isotropy family equal to  $\cJ_p=\Iso(\VV(\cF_p))$, 
such that for each $H \in \cJ_p$ the fixed set $E^H$ is simply-connected and $p$-locally homotopy equivalent to a sphere $S(W)^H$, where $W = V_H^{\oplus k}$
 for some integer $k$ and for some $V_H \in \VV (\cF_p)$. 
\end{proposition}

\begin{proof}
We recall  a result of Jackowski, McClure and Oliver \cite[Proposition 2.2]{jackowski-mcclure-oliver1}: there exists a simply-connected, finite dimensional $G$-CW-complex $B$ which is $\bF_p$-acyclic and has finitely many orbit types with isotropy in the family of $p$-subgroups $\cF_p$ in $G$. The quoted result applies more generally to all compact Lie groups and produces a complex with $p$-toral isotropy (meaning a compact Lie group $P$ whose identity component $P_0$ is a torus, and $P/P_0$ is a finite $p$-group). For $G$ finite, the $p$-toral subgroups are just the $p$-subgroups. 
A direct construction for $B$ can also be given using \cite[Corollary 3.15, Theorem 8.10]{hpy1} to ensure that all the fixed sets $B^H$  have finitely-generated $\bZp$-homology.  The property that all the fixed sets are simply-connected is established in the proof.

We now apply \cite[Proposition 4.3]{unlu-yalcin3} to this $G$-CW-complex $B$ and to the given $\cF_p$-representation $\VV(\cF_p)$, to obtain a $G$-equivariant spherical fibration $\cZ \to B$ with fiber type $S(\VV (\cF_p)^{\oplus k})$ for some $k$, such that $\cZ$ is finite dimensional (see \cite[Section 2]{unlu-yalcin3} for necessary definitions). The resulting $G$-CW-complex $\cZ$ has the required properties. In particular, since $B$ is $\bF_p$-acyclic then for each $p$-subgroup $H$, the fixed point set $B^H$ will be also $\bF_p$-acyclic (and $B^H\neq\emptyset$). This means that the (extended) isotropy family of $\cZ$ is $\cJ _p=\Iso (\VV (\cF _p ))$ and for every $H \in \cJ_p$, the mod-$p$ homology of $E^H$ is isomorphic to the mod-$p$ homology of $S(V_H^{\oplus k})^H$ for some $k$. By taking further fiber joins if necessary, we can assume that $E^H$ is simply connected for all $H\in \cJ_p$. Hence $\cZ^H$ 
is  $p$-locally homotopy equivalent to a sphere.   
\end{proof}

We now let $R = \bZp$, and consider the finite dimensional chain complex $\uC{\cZ}$ of free $\RG_G(p)$-modules. By taking joins, we may assume that this complex has ``homology gaps" of length $> l(\G_G)$ as required for \cite[Theorem 6.7]{hpy1}, and that all the non-empty fixed sets  $\cZ^H$ have $\un{n}(H) \geq 3$ and trivial action of $W_G(H)$ on homology. Let $\un{n}\up \colon  \cJ_p \to \bZ$ denote the dimension function $\HomDim \uC{\cZ}$.
 
The following result applies to chain complexes over $\RG_G$ with respect to any family $\cF$ of subgroups. 

\begin{lemma}\label{lem:dold}
Let $R$ be a noetherian ring and $G$ be a finite group.  Suppose that $\CC$ is an $n$-dimensional chain complex of projective $\RG_G$-modules with finitely generated homology groups. Then $\CC$ is chain homotopy equivalent to a finitely-generated projective $n$-dimensional chain complex 
over $\RG_G$.
\end{lemma}

\begin{proof} 
Note that the chain modules of $\CC$ are not assumed to be finitely-generated, but $H_i(\CC) = 0$ for $i > n$.
We first apply Dold's ``algebraic Postnikov system" technique  \cite[\S 7]{dold1}, to chain complexes over the orbit category (see \cite[\S 6]{hpy1}). 

According to this theory, given a positive projective chain complex $\CC$, there is a sequence of positive projective
chain complexes $\CC(i)$ indexed by positive integers such that $f\colon  \CC\rightarrow \CC(i)$ induces a homology isomorphism for
dimensions  $\leq i$. Moreover, there is a tower of
maps $$\xymatrix{&\CC(i)\ar[d]&\\
&\CC(i-1)\ar@{-->}[d]\ar[r]^{\alpha_{i}\ \ }&\Sigma ^{i+1}\PP (H_i)\\
\CC\ar[dr]\ar[ur]\ar[uur]\ar[r]&\CC(1)\ar[d]\ar[r]^{\alpha_2}&\Sigma ^3\PP (H_2)\\
&\CC(0)\ar[r]^{\alpha_1}&\Sigma ^2\PP (H_1)}$$ 
such that
$\CC(i)=\Sigma ^{-1} \CC(\alpha_i)$, where $\CC(\alpha_i)$ denotes
the algebraic mapping cone of $\alpha_i$, and $\PP(H_i)$ denotes a projective resolution of the homology module $H_i=H_i(\CC)$. 

By assumption, since the homology modules $H_i$ are finitely-generated and $R$ is noetherian, we can choose the projective resolutions $\PP(H_i)$ to be finitely-generated in each degree. Therefore, at each step the chain complex $\CC(i)$ consists of finitely-generated projective $\RG_G$-modules, and $\CC(n) \simeq \CC$ has homological dimension $\leq n$. 
  Now, since $H^{n+1}(\CC(n); M) = H^{n+1}(\CC; M) = 0$, 
 for  any $\RG_G$-module $M$, we conclude that $\CC(n)$ is chain homotopy equivalent to an $n$-dimensional finitely-generated projective chain complex by \cite[Prop.~11.10]{lueck3}. 
\end{proof}

\begin{remark} See \cite[11.31:ex.~2]{lueck3} or \cite[Satz 9]{tomDieck1981} for related background and previous results.
\end{remark}
Recall that  a dimension function $\un{n}$ has the  \emph{closure property} if it  satisfies condition (iii) of Definition \ref{def:algrep}. 

\begin{lemma}\label{lem:modelp}  If $\un{n}\up$ has the closure property,  then the chain complex $\uC{\cZ}$ is chain homotopy equivalent to an oriented $R$-homology $\un{n}\up$-sphere $\CC\up (0)$, which is an algebraic homotopy representation. 
\end{lemma}

\begin{proof} The chain complex $\uC{\cZ}$ is finite dimensional and free over $\RG_G$, but may not be finitely-generated. However, by the conclusion of Proposition \ref{pro:p-local sphere}, the homology groups $H_*(\uC{\cZ})$ are finitely generated since $\uC{\cZ}$ is an $R$-homology $\un{n}$-sphere. The result now follows from Lemma \ref{lem:dold}, which produces a finite length projective chain complex $\CC\up(0)$ of finitely-generated $\RG_G(p)$-modules. Note that $\uC{\cZ}$ satisfies the conditions (i)-(iii) in Definition \ref{def:algrep}, so $\CC\up (0)$  also satisfies these conditions (which are chain homotopy invariant), hence $\CC \up (0)$ is an algebraic homotopy representation. 
\end{proof}

Note that if $\un{n}\up$ satisfies the closure property, 
then $\CC\up (0)$ is an algebraic homotopy representation, meaning that it satisfies the condition (i), (ii), and (iii) in Definition \ref{def:algrep}, even though $\Dim \CC\up (0) $ 
may not be equal to $\un{n}\up  = \HomDim \CC\up (0)$. 

By taking joins, we may assume that there exists a common dimension 
$N = \un{n}\up(1)$, at $H=1$, for all $p \in \PG$. Moreover, we may assume that $N+1$ is a multiple of any given integer $m_G$ (to be chosen below). We now obtain the ``global" dimension function 
$$\un{n} = \bigcup\, \{\un{n}\up \vv p \in \PG \} \colon \cJ \to \bZ,$$
 where $\un{n}\up = \HomDim \CC\up(0)$, for all $p \in \PG$, and $\un{n}(1) = N$.

\begin{definition}[\emph{Preliminary local models}]\label{def:prelocal}  Let $\PG = \{ p \vv \rk_p G \geq 2\}$, and let $m_G$ denote the least common multiple of the $q$-periods for $G$  (as defined in \cite[p.~267]{swan1}), over all primes $q$ for which $\rk_q G = 1$. We assume that $\un{n}(1)+1$ is a multiple of $m_G$.
\begin{enumerate}
\item We will take the chain complex $\CC\up(0)$ constructed in Lemma \ref{lem:modelp} for our preliminary model at each prime $p \in \PG$.
\item If $\rk_q G = 1$, we take $\cJ_q = \{1\}$ and  $\CC\uq (0)$ as the $\RG_G$-chain complex $E_1 \PP$ where $\PP$ is a periodic resolution of $R$ as a $RG$-module with period $\un{n} (1)+1$ (for more details, see the proof of Theorem \ref{thm:q-local sphere} below, or \cite[Section 9B]{hpy1}).
\end{enumerate}
\end{definition}

This completes the construction of the preliminary local models at each prime dividing the order of $G$,
for a given family of $\cF_p$-representations.
In the next section we will modify these preliminary models to get $p$-local chain complexes $\CC\up$ over 
$\RG_G$ which are $R$-homology $\un{n}$-spheres for the dimension function $\un{n}$ described above.

\begin{example}\label{ex:peffective}
In the proof of Theorem A we will be using the setting of Example \ref{ex:pgroupfamily}. Suppose that $G$ is a rank two finite group which does not $p'$-involve $\Qd(p)$, for any odd prime $p$. We let $\PG$ be the set of primes $p$ where $\rk_p G = 2$. Under this condition, a result of Jackson \cite[Theorem 47]{jackson1} asserts that $G$ admits a $p$-\emph{effective} character $V_p$. 
Recall that ``$p$-effective" 
means that the restriction $\Res_E V_p$  to a rank two elementary abelian $p$-subgroup 
$E$  has no trivial summand. 
This guarantees that the set of isotropy subgroups $\cJ_p =\Iso (S(V_p ))$ consists of the rank one $p$-subgroups. In this setting, our preliminary local models  arise from the following special case when $p$ is odd: 

\begin{corollary}\label{cor:p-local sphereA} Let $p$ be an odd prime and $G$ be a finite rank two group with $\rk_p G = 2$. If $G$ does not $p'$-involve $\Qd(p)$, then there exists a simply-connected,  finite-dimensional $G$-CW-complex $\cZ$ with rank one $p$-group isotropy, which is $p$-locally homotopy equivalent to a sphere. 
\end{corollary}
\end{example}

Note that when $G$ is a $p$-group of rank two, then it has a central element $c$ of order $p$ in $G$. Using the subgroup generated by $c$, we can define the induced representation $V=\Ind _{\la c\ra } ^G \chi$ where $\chi$ is a nontrivial one dimensional complex representation of $\la c\ra$. Then, the $G$-action on $S(V)$ will satisfy the conclusion of the above corollary. It is proved by Dotzel-Hamrick \cite{dotzel-hamrick1} that all $p$-group actions on mod-$p$ homology spheres resemble linear actions on spheres.

\section{Construction of the local models: adding homology}
\label{sec:other primes}

Let $G$ be a finite group and let $\PG = \{ p \vv \rk_p G \geq 2\}$. 
 We recall the notation $\cJ_p = \Iso(\VV(\cF_p))$, for $p \in \PG$,  from  Definition \ref{def:allprimes}.  
 For $p \not \in \PG$ set 
  $\cJ_p=\{1\}$.  We will continue to work over the orbit category $\G_G = \Or_{\cJ}G$ where $\cJ = \bigcup\{ \cJ_p \vv p \in \PG\}$, or over its full subcategory $\G_G(p)$ with respect to the family $\cJ_p$.  
  For each prime $p$ dividing the order of $G$, let $\CC\up (0)$ denote the preliminary $p$-local model given in Definition \ref{def:prelocal},
and denote the homological dimension function of $\CC\up (0)$ by $\un{n}\up\colon \cJ_p \to \bZ$ for all primes dividing the order of $G$. In order to carry out this construction, we need to assume that each dimension function $\un{n}\up$ has the closure property. 
 
We now fix a prime $q$ dividing the order of $G$, 
and let $R = \bZq$. In Theorem \ref{thm:q-local sphere}, we will show how to add homology to the preliminary local model $\CC\uq(0)$, to obtain  
an algebraic homotopy representation  with dimension function $\un{n}\up \cup \un{n}\uq$ for any prime $p \in \PG$ such that $p\neq q$.
After finitely many such steps, we will obtain our local model $\CC\uq$ over $\RG_G$ with dimension function  
$$\HomDim \CC\uq = \un{n} =  \bigcup\, \{\un{n}\up \vv p \in \PG \}. 
$$
  The main result of this section is the following:

\begin{theorem}\label{thm:q-local sphere} Let $G$ be a finite group  and let $R = \bZq$. Suppose that $\CC$ is an algebraic homotopy representation over $R$, such that:
\begin{enumerate}
\item  $\CC$ is  an (oriented)  $R$-homology $\un{n}\uq$-sphere of projective $\RG_G(q)$-modules; 
\item If $1 \neq H \in \cJ_p$, then  $\rk_q(N_G(H)/H) \leq 1$, for every prime $p\neq q$;
\item The dimension function $\un{n}$ has  the closure property. 
\end{enumerate}  
Then there exists an  algebraic homotopy representation $\CC\uq$ over $R$,  which is an (oriented) $R$-homology $\un{n}$-sphere over $\RG_G$.
\end{theorem}

\begin{remark}\label{rem:necessity}
 Note that if there exists a $q$-local model $\CC\uq$ with isotropy in $\cJ_p \cup \cJ_q$, where $p \in \PG$, then for every $p$-subgroup $1 \neq H\in \cJ_p$, the $RN_G(H)/H$ complex $\CC\uq(H)$ is a finite length chain complex of finitely generated  modules which has the $R$-homology of an $\un{n}(H)$-sphere. 
 
 Since $R=\bZq$,  if we take a $q$-subgroup  $Q \leq N_G(H)/H$  with $H \neq 1$, and restrict  $\CC\uq(H)$ to $Q$, we obtain a finite dimensional \emph{projective} $RQ$-complex (see \cite[Lemma 3.6]{hpy1}). 
 This means $Q$ has periodic group cohomology and therefore it is a rank one subgroup. So, the condition (ii) in  Theorem \ref{thm:q-local sphere} is a necessary condition.
\end{remark} 

In order to carry out the construction in Theorem \ref{thm:q-local sphere},
  we also assume that  $\un{n}(H)+1$ is  a  multiple of the $q$-period of $W_G(H)$, for every $1 \neq H \in \cJ_p$, and that the gaps between non-zero homology dimensions are large enough:  more precisely,  for all $K, L\in \cJ$ with $\un{n}(K)> \un{n}(L)$, we have $\un{n}(K)-\un{n}(L)\geq l(\G _G)$,  where $l(\G_G)$ denotes the length of the longest chain of maps in the category $\G_G$.
We can easily guarantee both of these conditions  by taking joins of the preliminary local models we have constructed.

\begin{proof}[The proof of Theorem \ref{thm:q-local sphere}]
We obtain the complex $\CC ^{(q)}$ by adding homology specified by the dimension function $\un{n}\up$ step-by-step for each prime $p \in \PG$ with $p\neq q$. Let $p$ be a fixed prime with $p\neq q$. Assume that we have already added homology to the preliminary model and obtained a complex $\CC$ such that  $$\HomDim \CC = \un{n}^{(q)}\cup \bigcup\, \{\un{n}^{(r)} \vv r<p \text{ and } r \in \PG\}.$$ Now we will add more homology to  $\CC$ specified by the dimension function $\un{n}\up$ at the prime $p$. We will add these homologies 
by an inductive construction using the number of nonzero homology dimensions. Here is an outline of  the  argument: 

\begin{enumerate}
\item 
The starting point of the induction is the given complex $\CC$. Let $n_1> n_2>  \dots > n_s$ denote the set of dimensions $\un{n} (H)$, over all $H \in \cJ_p$. 
Note that, since the dimension function $\un{n}$ comes from a unitary representation, we have $n_s\geq 1$. 
 Let us denote by $\cF_i$, the collection of subgroups $1 \neq H\in \cJ_p$ such that $\un{n} (H)=n_i$. 

\item
Suppose that we have already added some homology to the given complex so that at this stage we have a finite projective chain complex $\CC$ over $\RG_G$, satisfying the conditions (i)-(iii) of Definition \ref{def:algrep}, which has the property that $\HomDim \CC (H)=\un{n}(H)$ for all $H \in \cF _{\leq k}$ where $\cF _{\leq k}=\bigcup _{i\leq k} \cF _i$. 
Our goal is to construct a new finite dimensional projective complex $\DD$ which also satisfies the conditions (i)-(iii) of Definition \ref{def:algrep}, and has the property that $\HomDim \DD (H)=\un{n}(H)$ for all $H \in \cF_{i}$ with $i \leq k+1$. 

\item
We will construct the complex $\DD$ as an extension of $\CC$ by a finite projective chain complex, whose homology is isomorphic to the homology that we need to add. Note that since the constructed chain complex $\DD$ must satisfy the conditions (i)-(iii), the homology we need to add should satisfy the condition that for every $H \leq K$ with $H, K \in \cF_{k+1}$, the restriction map on the added homology module is an $R$-homology isomorphism. 
\end{enumerate}

We will now begin the actual argument with the following useful notation.

\begin{definition}
Let $J_i$ denote the $\RG_G$-module which has the values $J_i (H)=R$ for all $H \in \cF_i$, and otherwise $J_i (H)=0$. The restriction maps  $r_H^K\colon J_i (K)\to J_i (H)$ for every $H, K \in \cF_i$ such that $H \leq K$, and the conjugation maps $c^g: J_i (K) \to J_i (\leftexp{g}{K}) $ for every $K\in \cF$ and $g\in G$, are assumed to be the identity maps (see \cite[\S 2]{h-yalcin3} for more details on these maps). 
\end{definition} 

In this notation,  the chain complex $\DD$ must have homology isomorphic to $J_{i}$ in dimension $n_i$ for all $i\leq k+1$, and in dimension zero the homology of $\DD$ should be isomorphic to $\un{R}$ restricted to $\cF_{k+1}$. It is in general a difficult problem to find projective chain complexes whose homology is given by a block of $R$-modules with prescribed restriction maps.  But in our situation we will be able to do this using some special properties of the poset of subgroups in $\cF_i$ coming from the closure property of $\un{n}$. 
Observe that we have the following property by Corollary \ref{cor:max}:

\begin{lemma}\label{lem:maxset}
For $1\leq i \leq s$, each poset $\cF_i$ is a disjoint union of components where each component has a unique maximal subgroup up to conjugacy.
\end{lemma}

\begin{proof} Follows from Corollary \ref{cor:max}. 
\end{proof}

For every $K \in \cJ_p$, the $q$-Sylow subgroup of the normalizer quotient $W_G(K)=N_G(K)/K$  has $q$-rank equal to one, hence it is $q$-periodic. By our starting assumption, the $q$-period of $W_G(K)$ divides $\un{n} (K)+1$. So by Swan \cite{swan1}, there exists a periodic projective resolution $\PP$ with
$$0 \to R \to P_n \to \dots \to P_1 \to P_0 \to R \to 0$$
over the group ring $RW_G(K)$ where $n=\un{n}(K)$. Note that this statement includes the possibility that $q$-Sylow subgroup of $W_G(K)$ is trivial since in that case $R$ would be projective as a $RW_G(K)$-module, and we can easily find 
a chain complex of the above form by adding a split projective 
chain complex. 

Now suppose that $K \in \cJ_p$ is such that $(K)$ is a maximal conjugacy class in $\cF_{k+1}$. Consider the $\RG_G$-complex $E_K \PP$ where $E_K$ denotes the extension functor defined in \cite[Sect.~2C]{hpy1}. By definition
$$E_K(\PP)(H) = \PP \otimes_{R[W_G(K)]}  R[(G/K)^H] $$
for every $H \in \cF$. 
We define the chain complex $E_{k+1} \PP$ as the direct sum of the chain complexes $E_K \PP$ over all representatives of isomorphism classes of maximal elements in $\cF_{k+1}$.
 Let $\NN$ denote the subcomplex of $E_{k+1}(\PP)$ obtained by restricting $E_K (\PP)$ to subgroups $H \in \cF_{\leq k}$. Let $I _{k+1} \PP$ denote the quotient complex $E_{k+1} (\PP )/\NN$. We have the following:

\begin{lemma} The homology of $I_{k+1} \PP$ is isomorphic to $J_{k+1}$ at dimensions $0$ and $n_{k+1}$ and zero everywhere else.
\end{lemma}

\begin{proof} The homology of $I_{k+1} \PP$ at $H\in \cF _{k+1}$ is isomorphic to $$\bigoplus \{ R \otimes _{R[W_G(K)]} R[(G/K)^H] : (K) \text{\ maximal in\ }  \cF _{k+1}\}$$ 
 at dimensions $0$ and $n_{k+1}$ and zero everywhere else 
 (since $N(H) = 0$ for $H \in \cF_{k+1}$).
Note that $(G/K)^H =\{ gK \colon  H^g \leq K\}$. If $gK$ is such that $H^g \leq K$, then $H\leq \leftexp{g}K$. Now by condition (iii), we must have $\langle K, \leftexp{g}K\rangle \in \cF _{k+1}$. But $(K)$ was a maximal conjugacy class in $\cF_{k+1}$, so we must have $K=\leftexp{g}K$, hence $g \in N_G(K)$.
This gives $1 \otimes gK = 1 \otimes 1$ in  $R \otimes _{R[W_G(K)]} R[(G/K)^H]$. Therefore 
$$R \otimes _{R[W_G(K)]} R[(G/K)^H]\cong R$$ for every $H \in \cF _{k+1}$. 
In addition,  $H$ can not be included in two  non-conjugate maximal subgroups in $\cF_{k+1}$, and therefore
$I_{k+1}(\PP)(H)\cong  R$ for all $H \in \cF_{k+1}$. Since the restriction maps are given by the inclusion map of fixed point sets $(G/H)^U \hookrightarrow (G/H)^V$ for every $U, V\in \cF _{k+1}$ with $V \leq U$, we can conclude that all restriction maps are identity maps. This completes the proof of the lemma.
\end{proof}

The above lemma shows that the homology of $I_{k+1}\PP$ is exactly the $\RG_G$-module $J_{k+1}$ that we would like to add to the homology of $\CC$. To construct $\DD$ we use an idea similar to the idea used in \cite[Section 9B]{hpy1}. Observe that for every $\RG_G$-chain map $f  \colon  \NN \to
\CC$, there is a push-out
diagram of chain complexes
$$\xymatrix{ 0\ar[r]& \NN \ar[d]^{f}\ar[r]&E_{k+1 }
\PP \ar[d]\ar[r]&I_{k+1}\PP \ar@{=}[d]\ar[r]&0\\
0\ar[r]&\CC \ar[r]&\CC _f \ar[r]&I_{k+1}\PP\ar[r]&0 \ .}$$

The homology of $\NN$ is only nonzero at dimensions $0$ and $n_{k+1}$ and at these dimensions the homology is only nonzero at subgroups $H \in \cF_{\leq k}$. At these subgroups the homology of $\NN (H)$ is isomorphic to the 
 direct sum 
 of modules of the form 
 $R \otimes _{RW_G(K)} R[(G/K)^H ]$, over $(K)$ maximal in $\cF_k$.
  Note that for every $H \in \cF_{\leq k}$, there is an augmentation map 
  $$\varepsilon _K ^H \colon R \otimes _{RW_G(K)} R[(G/K)^H ]\to R$$
   which takes $r \otimes gK$ to $r$ for every $r\in R$. The collection of these maps over all $H \in \cF_{\leq k}$ 
   gives a map of $\RG _G$-modules denoted 
   $$\varepsilon_K\colon (E_K R) _{\leq k} \to \un{R}_{\leq k}$$
    where the subscript $\leq k$ means the modules in question are restricted to $\cF_{\leq k}$. 
  Taking the sum over all conjugacy classes of maximal subgroups, we get a map  $$\varepsilon _{k+1}:= \sum _K \varepsilon _K\colon \ \bigoplus _K \ (E_K R )_{\leq k} \to \un{R}_{\leq k}.$$

In this notation, we have isomorphisms $H_0(\NN) \cong \bigoplus _K \ (E_K R )_{\leq k}$ and $H_0(\CC) \cong  \un{R}_{\leq k}$ which we will use to identify the homology groups in dimension zero.

\begin{lemma}\label{lem:rightH_0} If $f\colon  \NN \to \CC$ is a chain map such that the induced map 
$f_*\colon H_0 (\NN ) \to H_0 (\CC)$ agrees with the map $\varepsilon _{k+1}$  after the identifications above, then $H_0(\CC_f) \cong \un{R}_{\leq k+1}$.
\end{lemma}

\begin{proof} This follows from a commuting diagram argument which was also used in \cite[Section 9B]{hpy1} for a similar result.
Applying the zero-th homology functor, we obtain
$$\xymatrix{ 0\ar[r]& H_0(\NN) \ar[d]^{f}\ar[r]&H_0 (E_{k+1}
\PP ) \ar[d]\ar[r]&H_0(I_{k+1} \PP ) \ar@{=}[d]\ar[r]&0\\
0\ar[r]&H_0(\CC) \ar[r]&H_0 (\CC _f) \ar[r]&H_0(I_{k+1} \PP )\ar[r]&0 \ .}$$

The rows are still exact because $H_1(I_{k+1} \PP)(H)$ is non-zero only when $H\in \cF_{k+1}$, and both $H_0 (\NN)(H)$ and $H_0 (\CC)(H)$ are zero for $H \not \in \cF _{\leq k}$. So the connecting homomorphisms on the long exact homology sequences are zero maps.
Putting the modules we calculated before, we obtain
$$\xymatrix{& \ker \varepsilon_{k+1} \ar[d]\ar@{=}[r] & \ker \varepsilon _{k+1} \ar[d] \\   0\ar[r]& \oplus_K (E_KR)_{\leq k} \ar[d]^{\varepsilon _{k+1}}\ar[r]&\oplus _K E_K R \ar[d]\ar[r]& J_{k+1} \ar@{=}[d]\ar[r]&0\\
0\ar[r]&\un{R}_{\leq k} \ar[r]&H_0 (\CC _f) \ar[r]& J_{k+1}\ar[r]&0 \ .}$$
Now consider the $\RG_G$-modules in the middle vertical sequence. We claim that the restriction map $r^L _H$ 
from a subgroup $L \in \cF_{k+1}$ to a subgroup $H \in \cF_{\leq k}$ is the identity map in the module $H_0 (\CC_f)$. Note that the restriction maps $r^L _H$
in the modules appearing in the middle vertical sequence are given as follows (for each summand $K$ of maximal subgroups in $\cF_{k+1}$):
$$\xymatrix{0\ar[r]& 0 \ar[d]^{r ^L _H}\ar[r]& R \otimes _{RW_G(K)} R[(G/K)^L] \ar[d]^{r ^L _H}\ar[r]^-{\cong}& R \ar[d]^{r^L _H}\ar[r]&0\\
0\ar[r]&\ker \varepsilon ^H _K  \ar[r]& R \otimes _{RW_G(K)} R[(G/K)^H] \ar[r]^-{\varepsilon ^H_K} & R\ar[r]&0 \ .}$$
It is easy to see from this diagram that the restriction map on the right most vertical line is the identity map because the restriction map in the middle is the linearization of the inclusion map $(G/K)^L \subset (G/K)^H$ of fixed sets.
\end{proof}

The above lemma shows that  the complex $\CC _f$ has the correct homology
 if we take $f\colon  \NN \to \CC$ as the chain map inducing 
  $\varepsilon _{k+1}$ on $H_0$.
Unfortunately, we can not take $\DD$ as $\CC _f$ since the complex $I_{k+1}\PP$ is not  projective in general, and neither is $\NN$. We note that finding a chain map 
$f\colon \NN\to \CC$ satisfying the given condition is not an easy task without projectivity
 (compare \cite[Section 9B]{hpy1}, where this complex was  projective). So we first need to replace the sequence $0 \to \NN \to E_{k+1} \PP \to I_{k+1} \PP \to 0$ with a sequence of projective chain complexes.

\begin{lemma}\label{lem:projectiveresolutions} There is a diagram of chain complexes where all the complexes $\PP', \PP'', \PP '''$ are finite projective chain complexes over $\RG_G$ and all the vertical maps induce isomorphisms on homology:
$$\xymatrix{ 0\ar[r]& \PP' \ar[d]\ar[r]& \PP''\ar[d]\ar[r]&\PP ''' \ar[d]\ar[r]&0\\
0\ar[r]&\NN \ar[r]&E_{k+1} \PP \ar[r]&I_{k+1}\PP\ar[r]&0 \ .}$$
\end{lemma}

\begin{proof}  Since $E_K\PP$ is a projective chain complex of length $n$, $E_{k+1}\PP$ is a finite projective chain complex. So, by \cite[Lemma 11.6]{lueck3}, it is enough to show that $\NN$ is weakly equivalent to a finite projective complex. For this first note that $\NN = \bigoplus \NN_K$ is
a direct sum of chain complexes $\NN _K$ where $\NN _K$ is 
the restriction of $E_K \PP$ to subgroups $H \in \cF_{\leq k}$. So it is enough to show that $\NN _K$ is weakly equivalent to a finite projective chain complex. To prove this, we will show that for each $i$, the $\RG_G$-module $\NN_i:=(\NN_K)_i$ has a finite projective resolution. The module $\NN_i$ is nonzero only at subgroups $H \in \cF_{\leq k}$  and at each such a subgroup, we have 
$$\NN_i (H)=(E_K \PP_i )(H)=\PP_i \otimes _{RW_G(K)} R[(G/K)^H].$$ 
So, as an $RW_G(H)$-module $\NN_i(H)$ is a direct summand of $R[(G/K)^H]$ which is isomorphic to 
$$\bigoplus \{ R\big [W_G (H) / W_{\leftexp{g}{K}} (H ) \big ] : K\text{-conjugacy classes of subgroups\ }H^{g} \leq K \}$$ 
as an $RW_G(H)$-module.
 Since $K$ is a $p$-group, these modules are projective over the ground ring $R$ because $R$ is $q$-local. 
So, for each $H \in \cF _{\leq k}$, the $RW_G(H)$-module $\NN_i(H)$ is projective. Now consider the map $$\pi\colon \oplus _H E_H \NN_i (H) \to \NN_i $$ induced by maps adjoint to the identity maps at each $H$. We can take $\oplus _H E_H \NN_i (H)$ as the first projective module of the resolution,  and consider the kernel $\ZZ_0$ of $\pi\colon \oplus _H E_H \NN_i (H) \to \NN_i$. Note that $\ZZ_0$ has smaller length and it also have the property that at each $L$, the $W_G(L)$ modules $\ZZ_0(L)$ are projective.
This follows from the fact that $R[(G/H)^L ]$ is projective
as a $W_G(L)$-module by the same argument we used above. Continuing this way, we can find a finite projective resolution for $\NN_i$ of length $\leq l(\G)$.
\end{proof}

Now it remains to show that there is a chain map $f\colon \PP'\to \CC$, such that the induced map 
$f_*\colon H_0 (\PP ')\cong H_0(\NN) \to H_0 (\CC)$ is given by $\varepsilon_{k+1}$. Recall that  $\varepsilon_{k+1} =\sum _K \varepsilon_K$   is the sum of augmentation maps
over the conjugacy classes of maximal subgroups $K$ in $\cF_{k+1}$. 
Then the complex $\DD$ will be defined as the push-out complex that fits into the diagram 
$$\xymatrix{ 0\ar[r]& \PP' \ar[d]^{f}\ar[r]& \PP'' \ar[d]\ar[r]& \PP''' \ar@{=}[d]\ar[r]&0\\
0\ar[r]&\CC \ar[r]&\DD \ar[r]& \PP'''\ar[r]&0 \ .}
$$
Since both $\CC$ and $\PP'''$ are finite projective chain complexes, $\DD$ will also be a finite projective complex. The fact that $\DD$ has the right homology follows from Lemma \ref{lem:rightH_0}. 

To construct $f\colon\PP' \to \CC$, 
first note that  the reduced homology of the  chain complex $\CC$ is zero below dimension $n_k$. By assumption on the gaps between nonzero homology dimensions, we have 
  $n_k \geq n_{k+1}+l(\G_G)\geq l (\PP ')$. 
So, starting with the map $\varepsilon_{k+1}$ at $H_0$, we can obtain a chain map as follows:
$$
\xymatrix{ \ar[r]& 0 \ar[r] & P'_m \ar[d]^{f_m} \ar[r]^-{\bd _m ^{P'}}& \cdots
\ar[r]&P'_0
\ar[d]^{f_0}\ar[r]& H_0 (\NN ) \ar[d]^{ \varepsilon _{k+1}}\ar[r]&0
\\\ar[r] & C_{m+1} \ar[r] & C_m \ar[r]^{\bd _m ^C } & \cdots\ar[r] & C_0\ar[r]&H_0 (\CC )\ar[r]&0 }
$$
where $m=l (\PP ')$. This completes the proof of Theorem \ref{thm:q-local sphere}. 
\end{proof}

\section{The Proof of Theorem A}\label{sec:main theorem}

In this section we establish our main  technique for constructing actions on homotopy spheres, based on a given collection of $\cF_p$-representations, for the primes $p\in \PG$, where 
$\cF_p$ denotes the family of all $p$-power order subgroups of $G$  (see Definitions  \ref{def:sylowrep}  and \ref{def:allprimes}). 
Theorem A stated in the introduction will follow from this theorem almost immediately once we use the family of $p$-effective characters constructed by M.~A.~Jackson \cite{jackson1}. The main technical theorem is the following:

\begin{theorem}\label{thm:maintech} Let $G$ be a finite group and let $\PG = \{p \vv \rk_p G \geq 2\}$. Suppose that $\VV(\cF_p)$ is a $\cF_p$-representation for $G$, with $\Iso(\VV(\cF_p)) = \cJ_p$, for each $p \in \PG$. Then there exists a finite $G$-homotopy representation $X$ with isotropy in $\cJ= \bigcup \{ \cJ_p \vv p \in \PG\}$ if and only if the following two conditions hold
\begin{enumerate}
\item If $p \in \PG$ and $1\neq H \in \cJ_p$, then we have $\rk _q (N_G(H)/H )\leq 1$ for every $q\neq p$;
\item The dimension function $\un{n}$ has the closure property. 
\end{enumerate}
\end{theorem}


\begin{remark}  The construction we give in the proof of Theorem \ref{thm:maintech} gives a simply-connected $G$-homotopy representation $X$, with $\dim X^H \geq 3$, for all $H \in \cJ$, whenever $X^H \neq \emptyset$. It also relates the dimension function of $X$ to the linear dimension functions
$\Dim S(V_H)$, for $V_H \in \bigcup \{ \VV(\cF_p) \vv p \in \PG\}$ in the following way: for every prime $p \in \PG$, there exists an integer $k_p>0$ such that for every $H \in \cF_p$, the equality $\dim X^H=\dim S(V_H ^{\oplus k_p} )^H$ holds.

\end{remark}


\begin{proof}[The proof of Theorem \ref{thm:maintech}] 
The closure property for the dimension function $\un{n}$ is a necessary condition  to construct a $G$-homotopy  representation \cite[II.10]{tomDieck2}. 
As we discussed in the previous section (see Remark \ref{rem:necessity}), the condition on the $q$-rank of $N_G(H)/H$ is also a necessary condition for the existence of such actions (see Lemma \ref{lem:ranktwoA}).  Recall that this condition is used in an essential way in the proof of Theorem \ref{thm:q-local sphere}. 

By the realization theorem (Theorem 
\ref{thm:tightness}), we only need to  construct  a finite free chain complex of $\ZG _G $-modules satisfying the conditions (i), (ii) and (iii) of Definition \ref{def:algrep}.  If we apply Theorem \ref{thm:q-local sphere} to the preliminary local model constructed in Section \ref{sec:prime p},  we obtain a finite projective complex $\CC\up $, over the orbit category $\bZp\G_G$ with respect to the family $\cJ$, for each prime $p$ dividing the order of $G$. In addition, $\CC\up $ is an oriented  $\bZp$-homology $\un{n}$-sphere, with the same dimension function $\un{n} = \HomDim \CC\up (0)$ coming from the preliminary local models. By construction, the complex $\CC\up $ satisfies the conditions (i), (ii) and (iii) of Definition \ref{def:algrep} for $R = \bZp$.

We may also assume that $\un{n}(H)\geq 3$ for every $H \in \cJ$, and that the gaps between non-zero homology dimensions have the following property: for all $K, L\in \cJ$ with $\un{n}(K)> \un{n}(L)$, we have $\un{n}(K)-\un{n}(L)\geq l(\G _G)$ where $l(\G_G)$ denotes the length of the longest chain of maps in the category $\G_G$.

To complete the proof of Theorem \ref{thm:maintech}, we first need to glue these complexes $\CC\up $ together to obtain an algebraic $\un{n}$-sphere over $\bZ\G_G$. By \cite[Theorem 6.7]{hpy1}, there exists a finite projective chain complex $\CC$ of $\bZ\G_G$-modules, which is a $\bZ$-homology $\un{n}$-sphere, such that $\bZp \otimes \CC$ is chain homotopy equivalent to the local model $\CC\up$, for each prime $p$ dividing the order of $G$. The complex $\CC$  has a (possibly non-zero) finiteness obstruction (see Lueck \cite[\S10-11]{lueck3}), but this can be eliminated by joins (see \cite[\S 7]{hpy1}).

After applying \cite[Theorem 7.6]{hpy1}, we may assume that $\CC$ is a finite free chain complex of $\bZ\G_G$-modules which is a $\bZ$-homology $\un{n}$-sphere. Moreover, $\CC$ is an algebraic homotopy representation: it satisfies the conditions 
(i), (ii) and (iii) of Definition \ref{def:algrep} for $R = \bZ$, since these conditions hold locally at each prime. 

We have now established all the requirements for Theorem \ref{thm:tightness}. For the family $\cF$ used in its statement, we use $\cF = \cJ$. For all $H \in \cF$, we have the condition   $\un{n}(H)\geq 3$. Now Theorem \ref{thm:tightness} 
 gives a finite $G$-CW-complex $X\simeq S^n$ with isotropy $\cJ$ such that $X^H$ is an homotopy $\un{n}(H)$-sphere  for every $H \in \cJ$.   
\end{proof}

Now we are ready to prove Theorem A.

\begin{proof}[The proof of Theorem A] Let $G$ be a rank $2$ finite group and let $\PG$ denote the set of primes with $\rk _p G=2$. Since it is assumed that $G$ does not $p'$-involve $\Qd(p)$ for any odd prime $p$, we can apply  \cite[Theorem 47]{jackson1}  and obtain a $p$-effective representation $V_p$, for every prime $p\in \PG$. If the dimension function satisfies the closure property, 
we will apply Theorem \ref{thm:maintech} to the $\cF_p$-representations $\VV(\cF_p)$ given by this collection 
$\{V_p\}$ (see Example \ref{ex:pgroupfamily}). Since $V_p$ is $p$-effective means that all isotropy subgroups in $\cH_p$ are rank one $p$-subgroups (see Example \ref{ex:peffective}), the isotropy is contained in the family $\cH$  of rank one $p$-subgroups of $G$, for all $p \in \PG$. We therefore
obtain a    $G$-homotopy representation with rank one isotropy in $\cH$.  

The only thing we need to show is that for every prime $p\in \PG$, the dimension function $\un{n}\up$ of the $\cF_p$-representation $\VV(\cF_p)$ satisfies the closure property. 

\begin{lemma}\label{lem:podd} If  $V_p$ is 
 the induced representation  $ \Ind_E^{G_p} W$, 
where $E=\Omega_1(Z(G_p))$ and $W$ is the reduced regular representation of $E$. Then $\un{n}\up$ has the closure property.
\end{lemma}
\begin{proof} Using the Mackey formula, it is easy to see that for every $p$-subgroup $K \leq G_p$, the dimension of a fixed subspace in $\Res ^{G_p }_K V_p$ depends only on the index of $K$ in $G_p$, provided that the dimension is nonzero. This implies that for any two distinct $p$-subgroups $L< K$ in $G$, with nonempty fixed points on $\VV (\cF_p)$, we have $\un{n}\up (L)\neq \un{n}\up (K)$. Therefore the closure property for $\un{n}\up$ is automatic.
\end{proof}

Lemma \ref{lem:podd} takes care of all the odd prime cases (see the construction in 
\cite[Proposition 27]{jackson1} and \cite[Theorem 35]{jackson1}). 
The only remaining cases occur when $p=2$ and the Sylow $2$-subgroup is either dihedral, semi-dihedral, or wreathed (see \cite[Proposition 39]{jackson1}).

As we show in Example \ref{ex:aseven}, it is possible that in these cases, the dimension function may fail to satisfy the closure property. However,  this can only happen if there are two rank one $2$-subgroups $H,K$ with $H \cap K\neq 1$ such that $\la H, K \ra$ is not a $2$-group. Because if $\la H, K\ra$ is a $2$-group, then all these subgroups must lie in a Sylow $2$-subgroup and the closure property will follow from the fact that the restriction of $\un{n} ^{(2)}$ to $G_2$ is the dimension function of a linear representation
$V_2$. Since we assumed that $G$ has the rank one intersection property when $\Omega_1(Z(G_2))$  is not strongly closed, the proof of Theorem A is complete. 
\end{proof}

\begin{proof}[The proof of Corollary B]
If $\rk_q(G)\leq 1$, then for every $p$-group $H$, we must have $\rk _q (N_G(H)/H)\leq 1$. So we can apply Theorem A to obtain Corollary B. 
\end{proof}

Note that the condition about $\Qd(p)$ being not $p'$-involved in $G$ is a necessary condition for the existence of 
actions of rank $2$ groups on finite CW-complexes $X \simeq S^n$ with rank one isotropy. 
The following argument is an easy extension of the one given by \" Unl\" u in \cite[Theorem 3.3]{unlu1}.

\begin{proposition}\label{pro:involveQdp} Let $p$ be an odd prime.
If $G$ acts with rank one isotropy on a finite dimensional complex $X$  with the mod-$p$ homology of a  sphere, then $G$ cannot $p'$-involve $\Qd (p)$.
\end{proposition}

\begin{proof} Suppose that $G$ has a normal $p'$-subgroup $K$ such that $\Qd(p)$ is isomorphic to a subgroup in $N_G(K)/K$. Let $L$ be subgroup of $G$ such that $K\triangleleft L\leq N_G(K)$ and $L/K\cong \Qd(p)$.
The quotient group $Q=L/K$ acts on the orbit space $Y=X/K$ via the action defined by $(gK)(Kx)=Kgx$ for every $g \in L$ and $x\in X$. 

We observe two things about this action. First, by a transfer argument \cite[Theorem 2.4, p.~120]{bredon1}, the space $Y$ has the mod $p$ homology of a sphere. Second, all the isotropy subgroups of the $Q$-action on $Y$ have $p$-rank $\leq 1$. To see this, let $Q_y$ denote the isotropy subgroup at $y\in Y$ and let $x\in X$ be such that $y=Kx$. It is easy to see that $Q_y=L_xK/K \cong L_x /(L_x\cap K)$. Since $K$ is a $p'$-group, this shows that $p$-subgroups of $Q_y$ are isomorphic to $p$-subgroups of the isotropy subgroup $L_x$. Since $L$ acts on $X$ with rank one isotropy, we conclude that $\rk _p (Q_y)\leq 1$ for every $y \in Y$.  

Now the rest of the proof follows from the argument given in \"Unl\"u  \cite[Theorem 3.3]{unlu1}. Let $P$ be a $p$-Sylow subgroup of $Q\cong \Qd(p)$. Then $P$ is an extra-special $p$-group of order $p^3$ with exponent $p$ (since $p$ is odd). Let $c$ denote a central element and $a$ a non-central element in $P$. Since the $P$-action on $Y$ has rank one isotropy subgroups, we have $Y^E=\emptyset$ for every rank two $p$-subgroup $E\leq P$. Therefore $Y^{\langle c\rangle} =\emptyset$ by Smith theory, since otherwise $P/\langle c\rangle \cong \bZ/p \times \bZ/p$  would act freely on $Y^{\langle c\rangle}$ which is a  mod $p$ homology sphere. Now consider the subgroup $E=\langle a, c\rangle$. Since $\langle a \rangle$ and $\langle c \rangle$ are conjugate  in $Q$, all cyclic subgroups of $E$ are conjugate. In particular, we have $Y^H =\emptyset$ for every cyclic subgroup $H$ in $E$. This is a contradiction, since $E$ cannot act freely on $Y$.
\end{proof}

\begin{remark} A shorter proof can be given using more group theory.   For a  finite group $L$, and a  normal $p'$-subgroup   $K$ of $L$,   there is an isomorphism\footnote{We thank Radha Kessar for this information.} between the $p$-fusion systems  $\cF_L(S)$ and  $\cF_{L/K}(SK/K)$, where $S$ is a $p$-Sylow subgroup of $L$. So if $L/K \cong \Qd(p)$, then $L$ has an extra-special $p$-group $P$ of order $p^3$ with exponent $p$ such that a central element $c\in P$ is conjugate to a non-central element $a\in P$. This leads to a contradiction in the same way as above.
\end{remark}

\section{Discussion and examples}\label{sec:examples}

We first discuss the rank conditions in the statement of Theorem A. 
Suppose that $X$ is a finite $G$-CW-complex. Recall that  
$\Iso(X) = \{ H \vv H \leq G_x \text{\ for some\ } x \in X\} $
 denotes the minimal family containing all the isotropy subgroups of the $G$-action on $X$. We call this the \emph{isotropy family}. Note that 
 $H \in \Iso(X)$  if and only if $X^H \neq \emptyset$. We say that $X$ has \emph{rank $k$ isotropy} if $\rk G_x \leq k$ for all $x \in X$ and there exists a subgroup $H$ with $\rk H = k$ and $X^H \neq \emptyset$.

\begin{lemma}\label{lem:ranktwoA}
 Let $G$ be a finite group, and let $X$ be a finite $G$-CW-complex with $X \simeq S^n$.
 \begin{enumerate}
\item  If $H$ is a maximal $p$-subgroup in $\Iso(X)$, then $\rank_p (N_G(H)/H )\leq 1$.
\item  If $X$ has prime power isotropy and $1 \neq H\in \Iso(X)$
 is a $p$-subgroup, with $X^H$ an integral homology sphere, then
$\rank_q (N_G(H)/H) \leq 1$, for all primes $q\neq p$. 
\end{enumerate} \end{lemma}
 
\begin{proof}
This follows from two basic results of 
P.~A.~Smith theory \cite[III.8.1]{bredon1}), which state (i) that the fixed set of a $p$-group action on a finite-dimensional mod $p$ homology sphere is again a mod $p$ homology sphere (or the empty set), and (ii) that $\bZ/p\times \bZ/p$ can not act freely on a finite $G$-CW-complex $X$ with the mod $p$ homology of a sphere.

For any prime $p$ dividing the order of $G$, let $H \in \Iso(X)$ denote a maximal $p$-subgroup with  $X^H \neq \emptyset$. 
For any $x \in X^H$, we have $H \leq G_x$ and if $g\cdot x = x$, for some $ g \in N_G(H)$ of $p$-power order, it follows that the subgroup $\la H, g\ra \leq G_x$. Since $H$ was a maximal $p$-subgroup in $\Iso(X)$, we conclude that $g \in H$. Therefore the $p$-Sylow subgroup of 
$N_G(H)/H$ acts freely on the fixed set $X^H$, which is a mod $p$ homology sphere, and hence $\rank_p (N_G(H)/H) \leq 1$. 

If $q \neq p$ and $H$ is  a $p$-subgroup in $\Iso(X)$,  then  any $q$-subgroup $Q$ of $N_G(H)/H$ must act freely on $X^H$ (since $x\in X^H$ implies $G_x$ is a $p$-group).  Since $X^H$ is assumed to be an integral homology sphere, Smith theory implies that $\rk_q (Q)\leq 1$.
\end{proof}

\begin{example}\label{ex:ExtraSpecial}
If $G$ is the extra-special $p$-group of order $p^3$, then the centre $Z(G) = \bZ/p$ can not be a maximal isotropy subgroup in $\Iso(X)$. On the other hand, we know that $G$ acts on a finite complex $X\simeq S^n$ with rank one isotropy: just take the linear sphere $S(\ind^G _{Z(G)} W)$ for some nontrivial one-dimensional representation $W$ of $Z(G)$. So we can not require that $G$ acts on $X\simeq S^n$ with $\Iso(X)$ containing  all rank one subgroups. 
\end{example}

For any prime $p$, we can  restrict the $G$-action on $X$ to a $p$-subgroup of maximal rank. This gives the following well-known conclusion. 
 
 \begin{corollary}\label{cor:ranktwo}
If $X$ is a finite $G$-CW-complex with $X \simeq S^n$ and rank $k$ isotropy, then $\rk_p G \leq k+1$, for all primes $p$.
 \end{corollary}

\begin{remark}\label{rem:necessity2}
These results help to explain the rank conditions in Theorem A. First, if we have rank one isotropy, then we must assume that $G$ has rank two. However, condition (ii) on the $q$-ranks of  normalizer quotients is not necessary in general for the existence of a finite $G$-CW complex homotopy equivalent to a sphere with rank one prime power isotropy (see Example \ref{ex:aseven} for $G=A_7$). 

In contrast, Lemma \ref{lem:ranktwoA}(ii) shows that in order to construct a  
 $G$-homotopy   representation (with prime power isotropy) the normalizer quotients must satisfy the $q$-rank conditions  at all $p$-subgroups $H$, with $q \neq p$,  for which $X^H \neq \emptyset$. 
It follows that
the corresponding condition (ii) in the setting of
Theorem \ref{thm:maintech} is in fact a necessary condition. Example \ref{ex:ExtraSpecial} shows that not every rank one $p$-subgroup $H$ must fix a point on $X$ even when $X$ is assumed to be a  $G$-homotopy representation. 
\end{remark}

In order to get a complete list of necessary conditions, we must have more precise control of the structure of the isotropy subgroups. It might also be possible to construct finite $G$-CW complexes $X \simeq S^n$ with rank one prime power isotropy, for which the fixed sets $X^H$ are not homotopy spheres. The work of Petrie \cite[Theorem C]{petrie1978} and  tom Dieck \cite[Theorem 1.7]{tomDieck1985} explores this direction, but it is not clear to us that their results  answer our question. 

An attractive open problem is the case of finite rank two groups of \emph{odd} order. Such groups admit $G$-representation spheres $S(W_p)$ for each prime $p \in \PG$, whose isotropy groups have $p$-rank one (see Adem \cite[5.29]{adem_2007}). These spheres $S(W_p)$ could be used as the preliminary $p$-local models, instead of the construction given in Section \ref{sec:prime p}, but one would still need to add and subtract homology  to obtain the same homological dimension function at all primes. At present, we only know how to complete this step (as in Section \ref{sec:other primes}) under the conditions  (ii) and (iii) of Theorem \ref{thm:maintech}. The problem is that these conditions may not always hold for the
representation spheres $\{S(W_p): p \in \PG\}$.

Now we discuss an application of Theorem A.

\begin{example}\label{ex:asix}
\emph{The alternating group $G = A_6$ admits a finite
 $G$-homotopy representation $X$ with rank one prime power isotropy}.
  This follows from Theorem A once we verify that $G$ satisfies the necessary conditions.
Note that $A_6$ has order $2^3\cdot 3^2\cdot 5=360$ so it automatically satisfies the condition about $\Qd(p)$,  since it can not include an extra-special $p$-group of order $p^3$ for an odd prime $p$. For the $q$-rank condition, note that $\PG=\{2,3\}$, so we need to check this condition only for primes $p=2$ and $3$. 
Here are some easily verified facts:

\begin{itemize}
\addtolength{\itemsep}{0.2\baselineskip}

\item A $2$-Sylow subgroup $P\leq G$ is isomorphic to the dihedral group
$D_8$, so all rank one $2$-subgroups are cyclic, and $\cH_2 = \{1, C_2, C_4\}$.

\item $N_G(C_2)=P$, and $\rk _3( N_G(C_2)/C_2)=0$.
\item $N_G(C_4)=P$ and  $\rk _3( N_G(C_4)/C_4)=0$.
\end{itemize}

Now, let $Q$ be a $3$-Sylow subgroup in $G$. Then $Q \cong C_3\times C_3$.
\begin{itemize}
\addtolength{\itemsep}{0.2\baselineskip}
\item  Any subgroup of order $3$ in $G$ is conjugate to $C_3^A=\langle (123)\rangle $ or $C_3^B=\langle (123)(456)\rangle$.
\item $|N_G(C^A_3)/C^A_3|=6$ and  $\rk_2 (N_G(C^A_3)/C^A_3)=1$. 
\item $|N_G(C_3^B)/C_3^B|=6$ and $\rk_2 (N_G(C^B_3)/C^B_3)=1$
\end{itemize}

 We conclude that condition (ii) of Theorem A holds for this group. Note that the rank one intersection property also holds since in $A_6$ the intersection of any two distinct $C_4$'s is trivial. 
\end{example}

\begin{remark} Note that by the criteria given in \cite[Lemma 5.2]{adem-smith1}, the group $A_6$ does not have a character which is effective on elementary abelian $2$-subgroups. On the other hand, the triple cover of $A_6$ is a subgroup of $SU(3)$, and hence acts linearly on a sphere with rank one isotropy
by results of Adem, Davis and \"Unl\"u \cite[2.6, 2.9]{adem-davis-unlu1} on
the  \emph{fixity} of faithful unitary representations. More generally, they show that if $G \subset U(n)$ has fixity $f$, then $G$ acts linearly with rank one isotropy
on $U(n)/U(n-f)$. If $G \subset SU(n)$, then $G$ has fixity at most $n-2$.
\end{remark}

We now give an example which does not admit a $G$-homotopy   representation with rank one isotropy of prime power order.

\begin{example}\label{ex:aseven} \emph{The alternating group $G=A_7$
does not admit a  finite  $G$-homotopy representation with rank one prime power isotropy.} 
The  order of $G$ is $2^3\cdot 3^2\cdot 5 \cdot 7$, so this group also automatically satisfies the $\Qd(p)$ condition. 
Here is a summary of the main structural facts:
\begin{itemize}
\addtolength{\itemsep}{0.2\baselineskip}
\item The $3$-Sylow subgroup  $Q\leq G$ is isomorphic to $C_3\times C_3$.
\item Any subgroup of order $3$ in $G$ is conjugate to $C_3^A=\langle (123)\rangle $ or $C_3^B=\langle (123)(456)\rangle$.
\item  The $2$-Sylow subgroup of $N_G(C^A_3)$ is isomorphic to $D_8$.
\item $|N_G(C^A_3)/C^A_3|=24$ and $\rk_2 (N_G(C^A_3)/C^A_3)=2$.
\item $N_G(C_3^B )\cong (C_3\times C_3) \rtimes C_2$ and $\rk_2 (N_G(C^B_3)/C^B_3)=1$.
\item $|N_G(C_2)|=24$, and $\rk _3( N_G(C_2)/C_2)=1$.
\item $N_G(C_4)\cong D_8$ and $\rk _3( N_G(C_4)/C_4)=0$
\end{itemize}
We see that $\PG = \{2,3\}$, and the rank condition in Theorem A fails for $3$-subgroups, since there is a cyclic $3$-subgroup $H = C_3^A$ with $\rk _2 (N_G(H)/H)=2$.  Instead we can try to apply Theorem \ref{thm:maintech} directly by choosing $2$-effective and $3$-effective characters whose dimension functions have the closure property. A suitable $3$-effective character does exist, but it is not possible to find a $2$-effective character whose dimension function has the closure property.

Since all involutions in $G=A_7$ are conjugate,  the subgroup 
$\Omega_1(Z(G_2))$ is not strongly closed. 
 To see that $A_7$ does not satisfy the rank one intersection property either, take $H=\la (1234)(56) \ra \cong C_4$ and $K=\la (1234) (57)\ra \cong C_4$. The intersection of these cyclic subgroups is $H\cap K =\la (13)(24)\ra\cong C_2$. But the subgroups generated by $H$ and $K$ is not a $2$-group.

By applying the Borel-Smith conditions, we can easily show that  if there existed a    $G$-homotopy representation $X$ with rank one isotropy, then its dimension function $\un{n}$ would satisfy $\un{n}(H)=\un{n}(K)=\un{n} (H \cap K)\neq -1$, where $H$ and $K$ are given above. But then $\la H, K \ra $ will also fix a point,  contradicting our requirement that $X$ have prime power isotropy. 
\end{example}

\begin{example}\label{ex:PSU(3,3)} \emph{The group $G=PSU_3(3)$
is not    $2$-regular, but admits an
orthogonal linear action with rank one prime power isotropy}. The  order of $G$ is $2^5\cdot 3^3\cdot 7$. Here is a summary of the main structural facts:

\begin{itemize}
\addtolength{\itemsep}{0.2\baselineskip}
\item The $3$-Sylow subgroup  $G_3$ is isomorphic to the extra-special $3$-group of order $27$ and exponent $3$.
\item There are two conjugacy classes of subgroups $C_3^A$ and $C_3^B$ of order $3$.
\item $N_G(C^A_3)$ is isomorphic to $G_3 \rtimes C_8$, and $N_G(C_3^B)\cong C_3\times S_3$. In particular, $\rk_2(N_G(H)/H)=1$ for every cyclic subgroup $H\leq G$ of order $3$.
\item Sylow $2$ subgroup of $G$ is the wreathed group $(C_4\times C_4)\rtimes C_2$. 
\item All involutions in $G$ are conjugate (so $\Omega_1(Z(G_2))$ is not strongly closed). 
\item If $t$ is an involution in $G$, then $C_G(t) \cong GU_2(3)$ of order $96$, so $\rk_3 (N_G(H))\leq 1$ for every rank one $2$-subgroup $H\leq G$.
\end{itemize}
The facts listed above show that $G$ satisfies the normaliser rank condition of Theorem A. Note that $G$ also satisfies the $\Qd(p)$ condition since it has two conjugacy classes of subgroups of order $3$ and the Sylow $7$-subgroup is cyclic. 

By direct calculations in the group $GU_2(3)$ it is possible to show that $G$ does not satisfy the rank one intersection property. To see this note that $G$ includes a subgroup $H\cong GU_2(3)$ as centraliser of an involution $t \in G$. Since $SU_2(3)=SL_2(3)$, $H$ has a normal subgroup $S$ isomorphic to $SL_2(3) \cong Q_8 \rtimes C_3$ with quotient group $C_4$. In fact we can choose an element $u$ of order $8$ in $GU_2(3)$ such that $t\in \la u\ra$ and $GU_2(3)=S\cdot \la u\ra$. The group $H$ has another normal subgroup $K$ of order $16$ with quotient group $S_3$. So it is possible to find two cyclic subgroups $T=\la u\ra$, $T'=\la u' \ra$ isomorphic to $C_8$ such that $KT/K$ and $KT'/K$ corresponds to different cyclic $2$-subgroups in $S_3$. This means $\la T, T'\ra$ is not a $2$-group, but $T\cap T'\neq 1$ since $t\in T\cap T'$. Hence the rank one intersection property does not hold for $G$ (this argument was provided by Ron Solomon).

However  $G=PSU_3(3)$ does admit an orthogonal linear action with rank one isotropy 
(see \cite[Theorem 1.7]{adem-smith1}). 
By direct calculations using the character table one can show that all the nontrivial isotropy subgroups are isomorphic to one of the groups in $\{ C_2, C_3, C_4, Q_8 \}$, which are all rank one groups of prime power. 
 \end{example}
%
%

\section{The proof of Theorem C}\label{sec:simple}
The finite simple groups of rank two are listed in Adem-Smith \cite[p.423]{adem-smith1} as follows:

$$PSL_2(q), \ q \geq 5;\  PSL_2({q^2}), \ q \text{\ odd \ };\  PSL_3(q), \ q \text{\ odd\ };$$
$$PSU_3(q), \ q \text{\ odd \ };\  PSU_3(4);\  A_7 \text{\ and\ } M_{11}$$
where $q$ denotes a prime. Extensive information about the maximal subgroups of these simple groups is provided in \cite{mitchell1}, \cite{\gls}. To prove Theorem C we will consider separate cases. Note that $G = A_7$ is done in Example \ref{ex:aseven}.

\nr{Case 1}{$G =PSL_2(q), q \geq 5$} The order of $G$ is $q(q^2-1)/2$ and the maximal subgroups of $G$ are listed in \cite[6.5.1]{\gls}. From this list it is easy to see that the $2$-Sylow subgroup of $G$ is a dihedral group and for odd primes the Sylow subgroups are cyclic (see also \cite[4.10.5]{\gls}). It follows that $\PG = \{2\}$ and $G$ is $\Qd(p)$-free at odd primes, so Corollary B applies. We only need to show that $G$ satisfies  the rank one intersection property. This follows from the fact that the centraliser $C_G(z)$ of an involution $z \in G$ is a dihedral group of order $D_{q-\delta}$ where $\delta=\mp 1$ and $\delta \equiv q \mod 4$ (see \cite[Lemma 3.1]{gorenstein-walter1}). This implies that $C_G(z)$ has a cyclic subgroup of index $2$ which contains every element of $C_G(z)$ of order greater than 2. Hence we can not have two distinct cyclic $2$-subgroups in $G$ with  nontrivial intersection unless they include each other. 

By inspecting the character table of $G$, and applying the criterion \cite[Lemma 5.2]{adem-smith1}, we see that $PSL_2(q)$, $q >7$, does not admit an orthogonal representation $V$ with rank one isotropy on $S(V)$.

\medskip\noindent
\nr{Case 2}{$G =PSL_2({q^2}), q \geq 3$}  We did $PSL_2(9) = A_6$ explicitly in Example \ref{ex:asix}. In general, 
the order of $G$ is $q^2(q^4-1)/2$ 
and the maximal subgroups are again listed in  \cite[6.5.1]{\gls}. The conditions on the normalizer quotients needed for Theorem A can be checked 
at the primes $\PG = \{2, q\}$ using the information in  \cite{\gls},   and \cite[Chap.~II]{huppert1}. The $2$-Sylow subgroups are dihedral \cite[4.10.5]{\gls}, and the $q$-Sylow subgroup $Q$ is elementary abelian of rank two \cite[6.5.1]{\gls} (with normalizer $N_G(Q)$ represented by the parabolic subgroup of upper triangular matrices). At the other primes $p \neq 2, q$, any $p$-Sylow subgroup is contained in a dihedral group, and hence cyclic (see \cite[II.8.27]{huppert1}). The rank one intersection property can be checked in a similar way as in Case 1.

\nr{Case 3}{$PSL_3(q), q\geq 3$}  We refer to \cite[\S 15]{mitchell1} or \cite[6.5.3]{\gls} for the maximal subgroups. Since $G$ contains $\Qd(p)$ for $p = q$, this series of groups is ruled out. An explicit embedding is given by the matrices:
$$\Qd(p) = \left \{ {\left [ \vcenter{\xymatrix@C-20pt@R-25pt{a&b&e\cr c&d&f\cr 0&0&1}}
\right ]} \  : \   ad-bc = 1\right \} $$
with entries in $\bF_q$.
 
 \nr{Case 4}{$G = PSU_3(q), q \geq 3$} 
The order of $G$ is 
 $(q^3+1)q^3(q^2-1)/d$, 
 where $d = (3, q+1)$, and the maximal subgroups are given in \cite[\S 16]{mitchell1} or \cite[6.5.3]{\gls}. In particular, $G$ contains an abelian subgroup of order $(q+1)^2/d$. 
  If $9 \mid (q+1)$, then $G$ contains $\Qd(3)$, hence is ruled out,  so we assume that $9 \nmid (q+1)$. 
  
  If $3 \mid (q+1)$, then the 
 $3$-Sylow subgroup of $G$ is elementary abelian of order 9.  If $r>3$ is an odd prime dividing $q+1$, then the $r$-Sylow subgroup is abelian of rank two, and order equal to  the $r$-primary part of $(q+1)^2$. 
 Finally, if $r$ is an odd prime not dividing $q+1$, then $r$ divides $(q^2 - q+1)$ or $r$ divides $q-1$, 
  and the $r$-Sylow subgroup of $G$ is cyclic (see \cite[6.5.3(c)]{\gls},  
  \cite[p.~228, 241]{mitchell1} for the list of subgroups, and \cite[II.10.12]{huppert1},  and \cite[\S 1]{onan1} for additional details about the structure). 
   In summary, $\PG= \{2, q\} \cup \{r \mid (q+1):  \ r \text{\ an odd prime}\}$.

It is easy to see that the normaliser rank condition fails if an odd prime $r$ divides $(q+1)$, since  $2 \mid (q+1)^2/d $. Note that this failure can not be avoided by choosing a different representation since there is only one conjugacy class of order $2$ and order $3$ elements. The reason that there is only one conjugacy class of subgroups of order $3$ follows from the fact that when $3 \mid (q+1)$, then $G$ includes $PSU_3(2)= (C_3 \times C_3) \rtimes Q_8$ as a subgroup  (see \cite[6.5.3(d)]{\gls}, \cite[p.~241]{mitchell1}) and in this group there is only one conjugacy class of subgroups of order $3$. Hence $G = PSU_3(q)$ does not admit a finite $G$-homotopy representation with rank one prime power isotropy if an odd prime $r$ divides $(q+1)$.
 
 In the only remaining case we have $q+1=2^n$. The $2$-Sylow subgroup $G_2$ of $G$ is  the wreathed group $(C_{2^n}\times C_{2^n})\rtimes C_2$ and all involutions in $G$ are conjugate (a good reference for the facts we need is \cite[Chap.~I]{alperin-brauer-gorenstein1}).
 The case $q=3$ was discussed in Example \ref{ex:PSU(3,3)}, so we will assume $q+1=2^n$ with $n\geq 3$. 
 By direct calculation it can be shown that in this case
$\PG=\{ 2, q \}$ and the normalizer rank condition holds for these primes. However, it turns out that the rank one intersection property fails in this case. It is possible to find a pair of cyclic $2$-subgroups, one of which is a $C_4$ which intersect nontrivially and the subgroup that generate is not a $2$-group.  
In fact, the rank one intersection property fails in general for $G=PSU_3(q)$, $q\geq 3$. Again this failure can not be avoided.

\begin{proposition}
The group $G=PSU_3(q)$ does not admit a finite $G$-homotopy representations with rank one prime power isotropy, for $q+1=2^n$ with $n\geq 3$. 
\end{proposition} 
\begin{proof}
   Suppose that $G$ admits a  finite  $G$-homotopy representation $X$ with dimension functions $\un{n}$ such that only isotropy subgroups are rank one prime power subgroups. 
   Let  $G_2$ be generated by $x,y, z$ satisfying the relations $x^{2^n}=y^{2^n}=z^2=1$ and $zxz=y$.
    We apply the Borel-Smith conditions to the subgroup lattice of $G_2$, starting with the quotient group $U/V = C_2 \times C_2$ with $U=\la xy^{-1}, x^{2^{n-1}}, z\ra$ and $V=\la xy^{-1}\ra$. This gives $\un{n}(V)=\un{n} (Q)$ where $Q=\la xy^{-1}, x^{2^{n-1}}z \ra$ is a generalized quaternion subgroup of $U$ that includes $V$. Applying the Borel-Smith conditions to the subquotients of the  dihedral subgroup of $U$ that includes $V$, we can obtain further that $\un{n} (Q)=\un{n} (H)=\un{n}(\la t\ra )\neq -1$
where $H$ is any nontrivial subgroup of $V$. 

To reach to a contradiction, we choose an involution $t\in G$ and consider its centraliser $C_G(t)$, which is isomorphic to the group $GU_2(q)$ (see \cite[Proposition 4, p.~21]{alperin-brauer-gorenstein1}). This group has $SU_2(q) \cong SL_2(q)$ as a subgroup, so in $C_G(t)$ there is a subgroup $S$ isomorphic to $SL_2(q)$. When $q>3$, the Sylow $2$-subgroup of $SL_2(q)$ is not normal, hence  in $S$ it is possible to find two conjugate quaternion subgroups $Q$ and $Q'$ such that $t\in Q\cap Q'$. Note that by monotonicity of the dimension function,  the equalities $\un{n}(Q)=\un{n} (Q')=\un{n}(\la t\ra)$ imply that $\un{n}(Q)=\un{n} (Q')=\un{n}(Q\cap Q')$. 

We claim that $\la Q, Q' \ra$ is not a $2$-group. This will prove that $\un{n}$ does not have the closure property, hence it will give a contradiction to the existence of the $G$-homotopy representation with rank one prime power isotropy. Suppose to the contrary that $\la Q, Q' \ra$ is a $2$-group. Then both $Q$ and $Q'$ will lie, as distinct subgroups, in a Sylow $2$-subgroup $G_2$ of $G=PSU_3(q)$. But the wreathed group $(C_{2^n}\times C_{2^n})\rtimes C_2$ has a unique quaternion subgroup of order $2^{n+1}$ (see \cite[Lemma 2(v), p.~9]{alperin-brauer-gorenstein1}). This is a contradiction.
\end{proof}

 \nr{Case 5}{$G = M_{11}$} The order of $G$ is $7920 = 2^4\cdot 3^2\cdot 5\cdot 11$ and $\PG = \{2, 3\}$. This group is $\Qd(p)$-free, but the $2$-rank $\rk_2 (N_G(H)/H)=2$, for $H$ a subgroup of order three (see  \cite[p.~262]{\gls}). 
 Since all the subgroups of order three are conjugate, $G$ does not admit a finite $G$-homotopy representation with rank one prime power isotropy.                                                     
 
\nr{Case 6}{$G = PSU_3(4)$} The order of $G$ is $65280 = 2^6 \cdot 3 \cdot 5^2 \cdot 13$ and $\PG = \{2, 5\}$. This group has a linear representation $V$ such that the $G$-action on $S(V)$ has rank one isotropy 
(see \cite[p.~425]{adem-smith1}). One can check that all
 the nontrivial isotropy subgroups are in the set $\{C_2, C_3, C_4, C_5\}$.
  

\providecommand{\bysame}{\leavevmode\hbox to3em{\hrulefill}\thinspace}
\providecommand{\MR}{\relax\ifhmode\unskip\space\fi MR }
\providecommand{\MRhref}[2]{%
  \href{http://www.ams.org/mathscinet-getitem?mr=#1}{#2}
}
\providecommand{\href}[2]{#2}

\end{document}